\documentclass[9pt]{IEEEtran}

\usepackage{graphicx,verbatim}
\usepackage{amsmath,amssymb,amsthm}
\usepackage{algorithm,algorithmic}
\usepackage[small,bf]{caption}
\usepackage{subcaption} 
\usepackage{cases}
\usepackage[table,dvipsnames]{xcolor}
\usepackage[shortlabels]{enumitem}
\usepackage{hyperref,url}
\usepackage{cite}
\hypersetup{
	colorlinks=true,
	linkcolor=red,     
	urlcolor=magenta,
	citecolor={blue},
}
\usepackage{booktabs,adjustbox,multirow}  
\usepackage{import} 
\usepackage{aligned-overset} 
\usepackage{dsfont,bm}
\usepackage{lmodern}
\usepackage{upgreek} 
\newtheorem{theorem}{{Theorem}}
\newtheorem{lemma}{{Lemma}}

\newtheorem{assumption}{{Assumption}}
\newtheorem{remark}{{Remark}}
\usepackage{fancyhdr}

\newcommand{\grad}{{\nabla}}   
\newcommand{\zero}{\mathbf{0}}  
\newcommand{\one}{\mathbf{1}}   
\newcommand{\real}{\mathbb{R}}  

\newcommand{\bdiag}{{\rm blkdiag}}  
\newcommand{\col}{\mathrm{col}}     

\newcommand{\Ex}{\mathop{\mathbb{E}{}}}

\newcommand{\minimize}{\mathop{\text{minimize}}}

\newcommand{\eg}{{\it e.g.}}
\newcommand{\ie}{{\it i.e.}}

\newcommand{\tran}{^{\textit{\footnotesize \texttt{T}}}} 
\renewcommand{\top}{\textit{\footnotesize \texttt{T}}} 
\newcommand{\define}{\triangleq} 

\newcommand{\qd}{\hfill{$\blacksquare$}}

\def\B{{\mathbf{B}}}

\def\F{{\mathbf{F}}}
\def\G{{\mathbf{G}}}

\def\I{{\mathbf{I}}}

\def\U{{\mathbf{U}}}
\def\V{{\mathbf{V}}}
\def\W{{\mathbf{W}}}

\def\a{{\mathbf{a}}}
\def\b{{\mathbf{b}}}

\def\f{{\mathbf{f}}}
\def\g{{\mathbf{g}}}

\def\v{{\mathbf{v}}}

\def\x{{\mathbf{x}}}

\def\z{{\mathbf{z}}}

\newcommand{\cE}{{\mathcal{E}}}
\newcommand{\cF}{{\mathcal{F}}}

\newcommand{\cN}{{\mathcal{N}}}
\newcommand{\cO}{{\mathcal{O}}}

\newcommand{\cU}{{\mathcal{U}}}

\def\evdots{\vbox{\baselineskip=2pt \lineskiplimit=0pt 
		\kern6pt \hbox{$.$}\hbox{$.$}\hbox{$.$}}}  
\pdfoutput=1
\allowdisplaybreaks

\begin{document}
	
	\title{\bfseries \huge An Enhanced  Gradient-Tracking Bound for Distributed Online Stochastic Convex Optimization}
	
	\author{Sulaiman A. Alghunaim and Kun Yuan 
		\thanks{ S. A. Alghunaim (\texttt{sulaiman.alghunaim@ku.edu.kw}) is with the Department of Electrical Engineering, Kuwait University,  Kuwait. K. Yuan (\texttt{kunyuan@pku.edu.cn}) is with the Center for Machine Learning Research, Peking University, China.}
	}
	\maketitle
	
	\begin{abstract}
		Gradient-tracking (GT) based decentralized methods have emerged as an effective and viable alternative method to decentralized (stochastic) gradient descent (DSGD) when solving distributed online stochastic optimization problems. Initial studies of GT methods implied that GT methods have worse network dependent rate than DSGD, contradicting experimental results. This dilemma has recently been resolved, and  tighter rates for GT methods have been established, which improves upon DSGD.

		In this work, we establish more enhanced rates for GT methods under the online stochastic convex settings. We present an alternative approach for analyzing GT methods for convex problems and over static graphs. When compared to previous analyses, this approach allows us to establish enhanced network dependent rates. 
	\end{abstract}
	

	\section{Introduction}
	We  consider the multi-agent consensus optimization problem, in which   $n$  agents work together  to solve the following stochastic optimization problem:
	\begin{align} \label{min_consenus}
		\minimize_{x \in \real^d} \quad f(x)= \frac{1}{n} \sum_{i=1}^n f_i(x) \quad f_i(x)\define \Ex [F_i(x;\xi_i)].
	\end{align}
	Here, $f_i: \real^d \rightarrow \real$ is the  private cost function held by agent $i$, which is defined as the expected value of some loss function $F_i(\cdot ,\xi_i)$ over local random variable $\xi_i$ (\eg, data points). An algorithm that solves \eqref{min_consenus} is said to be a {\em decentralized} method if its implementation requires the agents to communicate only with agents who are directly connected to them (\ie, neighbors) based on the given network topology/graph.
	
	One of the most popular decentralized methods to solve problem \eqref{min_consenus} is decentralized stochastic gradient descent (DSGD) \cite{lopes2008diffusion,ram2010distributed,
		cattivelli2010diffusion}. While DSGD is communication efficient and simple to implement, it converges slowly when the local functions/data are heterogeneous across nodes. Furthermore, because data heterogeneity can be amplified by large and sparse network topologies \cite{yuan2020influence}, DSGD performance is significantly degraded with these topologies. 
	
	In this work, we analyze the performance of the gradient-tracking method \cite{xu2015augmented,di2016next}, which is another well-known decentralized method that solves problem \eqref{min_consenus}. To describe the algorithm, we let
	$w_{ij} \geq 0$ denote the weight used by agent $i$ to scale information received from agent $j$ with $w_{ij}=0$ if $j \notin \cN_i$ where $\cN_i$ is the neighborhood of agent $i$.  The adapt-then-combine gradient-tracking (ATC-GT) method \cite{xu2015augmented} is described as follows:
	\begin{subequations} \label{GT_atc_alg}
		\begin{align} 
			x_i^{k+1}&=\sum_{j \in \cN_i} w_{ij} (x_j^{k} - \alpha g_j^{k}) \\
			g_i^{k+1} &= \sum_{j \in \cN_i} w_{ij} \big(g_j^{k} + \grad F_j(x_j^{k+1};\xi_j^{k+1})-\grad F_j(x_j^{k};\xi_j^{k}) \big)
		\end{align}
	\end{subequations}
	with initialization $g_i^0=\grad F_i(x_i^0;\xi^0_i)$ and arbitrary  $x_i^{0} \in \real^d$. Here, $\grad F_i(x_i^k;\xi^k_i)$ is the stochastic gradient and $\xi^k_i$ is the data sampled by agent $i$ at iteration $k$.
	
	Gradient-tracking can eliminate the  impact of heterogeneity between local functions \cite{xu2015augmented,di2016next,nedic2017achieving,qu2017harnessing}. In massive numerical experiments reported in \cite{pu2021distributed,xin2021improved,lu2019gnsd,yuan2021removing}, GT can significantly outperform DSGD in the online stochastic setting. Initial studies on the convergence rate of GT methods  are inadequate; they provide loose convergence rates that are more sensitive to network topology than vanilla DSGD. According to these findings, GT will converge slower than DSGD on large and sparse networks, which is counter-intuitive and contradicts numerical results published in the literature. Recent works \cite{koloskova2021improved,alghunaim2021unified}  establish the first convergence rates for GT that are faster than DSGD and more robust to sparse topologies under stochastic and non-convex settings. In this paper, we will provide additional enhancements for GT under convex and strongly convex settings.


	\subsection{Related works}
	
	Gradient-tracking (GT) methods, which utilize dynamic tracking mechanisms \cite{zhu2010discrete} to approximate the globally averaged gradient, have emerged as an alternative to decentralized gradient descent (DGD) \cite{nedic2009distributed,lopes2008diffusion,ram2010distributed,cattivelli2010diffusion,yuan2016convergence} with exact convergence for deterministic problems \cite{xu2015augmented,di2016next,nedic2017achieving,qu2017harnessing}. Since their inception, numerous works have investigated GT methods in a variety of contexts \cite{xi2018linear,pu2020push,daneshmand2018second,sun2019convergence,scutari2019distributed,saadatniaki2020decentralized,pu2021distributed,xin2021improved,tang2020distributed,xin2020fast,xin2021fast,li2020communication,sun2020improving}. However, all of these works provide convergence rates that can be worse than vanilla DSGD. In particular, these results indicate that GT is less robust to sparse topologies even if it can remove the influence of data heterogeneity.   
	The work \cite{alghunaim2021unified} established refined bounds for various methods including GT methods that improve upon DSGD under nonconvex settings. Improved network dependent bounds for GT methods in both convex and non-convex settings are also provided in \cite{koloskova2021improved}. In this work, we provide additional improvements over previous works in convex and strongly convex settings -- see Table \ref{table}.
	
	It should be noted that there are other methods that are different from GT methods but have been shown to have comparable or superior performance -- see \cite{alghunaim2019decentralized,alghunaim2021unified}  and references therein.  In contrast to these other methods, GT methods have been shown to converge in a variety of scenarios, such as directed graphs and time-varying graphs \cite{xi2018linear,pu2020push,scutari2019distributed}. We should also mention that there are modifications to GT approaches that can improve the rate at the price of knowing additional network information and/or more computation/memory \cite{sun2019convergence}. However, the focus of this study is on {\em basic vanilla} GT methods. 
	
	\begin{table*}[t]    \footnotesize
		\renewcommand{\arraystretch}{2}
		\begin{center}
			\caption{\footnotesize Convergence rate to reach $\epsilon$ accuracy. The strongly convex (SC)  and PL condition rates ignores iteration logarithmic factors. The quantity $\lambda=\rho(W-\tfrac{1}{n} \one \one\tran) \in (0,1)$ is the mixing rate of the network where $W$ is the network combination matrix.
				$a_0=	 \|\bar{x}^{0} -x^\star\|^2 $, $\varsigma_{\star}^2= \frac{1}{n} \sum_{i=1}^n \|\grad f_i(x^\star)\|^2$,  $\varsigma_{0}^2= \frac{1}{n} \sum_{i=1}^n \|\grad f_i(x^0)-\grad f(x^0)\|^2$,  $x^0$ is the initialization for all nodes, and $x^\star$ is an optimal solution of \eqref{min_consenus}.    } \label{table}
			\begin{adjustbox}{max width=2\columnwidth} 
				\begin{tabular}{llll}  \toprule
					\multicolumn{2}{c}{{\sc Reference}}	  & {\sc Iterations to $\epsilon$ accuracy} & {\sc Remark}
					\\ 
					\midrule 
					Convex &		  \cite{koloskova2021improved}   & 
					$\dfrac{1}{n \epsilon^2}+ \frac{\log(\frac{1}{1-\lambda})^{1/2}}{(1-\lambda)^{1/2}}   \dfrac{1 }{\epsilon^{3/ 2}}+\frac{\log(\frac{1}{1-\lambda}) (a_0+\varsigma_0^2)}{1-\lambda}  \dfrac{1}{\epsilon}$ 
					& Rate holds only when iteration number $K > \frac{\log(\frac{1}{1-\lambda})}{1-\lambda}$ \\
					\rowcolor[gray]{0.92}	Convex	&
					{\scriptsize	 \textbf{Our work}}    &
					$\dfrac{1}{n \epsilon^2}+ \frac1{(1-\lambda)^{1/2}}  \dfrac{1 }{ \epsilon^{3 / 2}}+\frac{(a_0+ \varsigma_\star^2)}{(1-\lambda) } \dfrac{1}{\epsilon}$ 
					&
					-- \\ \midrule
					
					SC  &
					\cite{pu2021distributed}  
					&
					$\dfrac{1}{n  \epsilon} +
					\frac{1}{(1-\lambda)^{3/ 2}} \dfrac{1 }{ \sqrt{\epsilon}}	+ \dfrac{C}{\sqrt{\epsilon}} $
					& $C$ depends on $1/(1-\lambda)$  \\
					PL$^*$   &  \cite{xin2021improved}  
					&
					$\dfrac{1}{n  \epsilon} +
					\frac{1}{(1-\lambda)^{3/ 2}} \dfrac{1 }{ \sqrt{\epsilon}}	+ \tilde{C} \log \frac{1}{\epsilon} $
					& $\tilde{C}$ depends on $1/(1-\lambda)$	 \\
					SC	&		  \cite{koloskova2021improved}   & 
					$\dfrac{1}{n \epsilon}+ \frac{\log(\frac{1}{1-\lambda})^{1/2}}{(1-\lambda)^{1/2}}   \dfrac{1 }{ \sqrt{\epsilon}}+\frac{\log(\frac{1}{1-\lambda})}{(1-\lambda) }  \log \left( \frac{  (a_0+\varsigma_0^2)}{(1-\lambda) \epsilon} \right)$
					&	 Rate holds only when iteration number $K > \frac{\log(\frac{1}{1-\lambda})}{1-\lambda}$	 \\
					PL$^*$  &  \cite{alghunaim2021unified}   &
					$\dfrac{1}{n  \epsilon}+
					\left(\frac{1}{(1-\lambda)^{1/2}} + \frac{1}{(1-\lambda) \sqrt{n}} \right) \dfrac{1}{\sqrt{\epsilon}}
					+ \frac{1}{1-\lambda} \log \left(\frac{(a_0+\varsigma_{\star}^2)}{\epsilon}\right)$
					&	Rate holds by tuning stepsize from \cite[Theorem 2]{alghunaim2021unified}	\\
					\rowcolor[gray]{.92}	SC	&		{\scriptsize \textbf{Our work} }  &
					$\dfrac{1}{n  \epsilon}+ \frac{1}{(1-\lambda)^{1/2}} \dfrac{1 }{\sqrt{\epsilon}}
					+ \dfrac{1}{1-\lambda} \log \left(\frac{(a_0+\varsigma_{\star}^2)}{\epsilon}\right)$
					&	
					-- \\  \bottomrule
				\end{tabular}
			\end{adjustbox}
			\label{table_transient_time_PL_bound}
		\end{center} 
		{\scriptsize $*$ The PL condition is weaker than SC and can hold for nonconvex functions; any SC function satisfies the PL condition.}
	\end{table*}

	\subsection{Contributions}
	\begin{itemize}
		\item We present an alternative approach for analyzing GT methods in convex and static graph settings, which may be useful for analyzing GT methods in other settings such as variance-reduced gradients.

		\item In {\em stochastic} and {\em convex} environments, our convergence rate improve and tighten existing {GT} bounds. We show, in particular, that under convex settings, GT methods  have better dependence on network topologies than in nonconvex settings \cite{alghunaim2021unified}. Also, our bounds removes the network dependent log factors in \cite{koloskova2021improved} -- See Table \ref{table}.

	\end{itemize}

	\section{ATC-GT and Main Assumption}
	In this section, we describe the GT algorithm \eqref{GT_atc_alg} in network notation and list all necessary assumptions. We begin by defining some network quantities.

	\subsection{GT in network notation}
	We define $x_i^k \in \real^d$ as the estimated value of $x \in \real^d$ at agent $i$ and iteration (time) $k$, and we introduce the augmented network quantities:
	\begin{align*}
		\x^k &\define \col\{x_1^k,\dots,x^k_n\} \in \real^{dn} \\
		\f(\x^k)&\define \sum_{i=1}^n f_i(x_i^k) \\
		\grad \f(\x^k)& \define \col\{\grad f_1(x_1^k),\dots,\grad f_n(x_n^k)\} \\
		\grad \F(\x^{k})&\define \col\{\grad F_1(x_1^k;\xi^k_1),\dots,\grad F_n(x_n^k;\xi^k_n)\} \\
		\g^k &\define \col\{g_1^k,\dots,g_n^k\} \in \real^{dn} .
	\end{align*}
	Here, $\col\{\cdot\}$ is an operation to stack all vectors on top of each other. In addition, we define
	\begin{align}
		W \define [w_{ij}] \in \real^{n \times n}, \quad 
		\W \define W \otimes I_d,
	\end{align}
	where $W$ is the network weight (or combination, mixing, gossip) matrix with elements $w_{ij}$,  and symbol $\otimes$ denotes the Kronecker product operation.  Using the above quantities, the ATC-GT method \eqref{GT_atc_alg} can be described as follows:
	\begin{subequations} \label{atc-gt}
		\begin{align}
			\x^{k+1}&=\W [\x^{k} - \alpha \g^{k}] \\
			\g^{k+1} &= \W [\g^{k} + \grad \mathbf{F}(\mathbf{x}^{k+1})-\grad \mathbf{F}(\mathbf{x}^{k}) ],
		\end{align}
	\end{subequations}
	with  initialization $\g^0=\grad \F(\x^0)$ and arbitrary  $\x^{0}$.

	\subsection{Assumptions}
	Here, we  list the assumptions used in our analyses. Our first assumption is on the network graph stated below.  
	
	\begin{assumption}[\sc \small Weight matrix] \label{assump:network} 
		The network graph is assumed to be static and, the weight matrix $W$ to be doubly stochastic and primitive. We further assume $W$ to be symmetric and positive semidefinite. \qd
	\end{assumption}
	
	\noindent It is important to note that assuming $W$ to be positive semidefinite is not restrictive; given  any doubly stochastic and symmetric $\tilde{W}$, we can easily construct a positive semidefinite weight matrix by $W = (I + \tilde{W})/2$. We also remark that, under Assumption \ref{assump:network}, the mixing rate of the network is:
	\begin{align} \label{graph_mixing_rate}
		\lambda \define  \big\|W-\tfrac{1}{n} \one \one\tran\big\| =\max_{i \in \{2,\ldots,n\}} |\lambda_i| <1.
	\end{align}

	The next assumption is on the objective function. 
	\begin{assumption}[\sc \small Objective function] \label{assump:smoothness} Each function $f_i: \real^d \rightarrow \real$ is $L$-smooth 
		\begin{align} \label{smooth_f_eq}
			\|\grad f_i(y) -\grad f_i(z)\| \leq L \|y -z\|, \quad \forall~y,z \in \real^d
		\end{align}
		and ($\mu$-strongly) convex
		for some $L \geq \mu  \geq 0$. As a result, the aggregate function $f(x)=\frac{1}{n} \sum_{i=1}^n f_i(x)$ is also $L$-smooth and ($\mu$-strongly) convex. (When $\mu=0$, then the objective functions are simply convex.)  \qd
	\end{assumption}

	We now state our final assumption related to the gradient noise.
	\begin{assumption}[\sc \small Gradient noise] \label{assump:noise} 
		For all $\{i\}_{i=1}^n$ and $k=0,1,\ldots$, we assume the following inequalities hold
		\begin{subequations} \label{noise_bound_eq}
			\begin{align}
				\Ex \big[\grad F_i(x_i^k;\xi_i^k)-\grad f_i(x_i^k) ~|~ \bm{\cF}^{k}\big] &=0, \label{noise_bound_eq_mean} \\
				\Ex \big[\|\grad F_i(x_i^k;\xi_i^k)-\grad f_i(x_i^k)\|^2 ~|~ \bm{\cF}^{k} \big] &\leq \sigma^2, \label{noise_bound_eq_variance}
			\end{align}
		\end{subequations}
		for some $\sigma^2 \geq 0$, where  $\bm{\cF}^k \define \{\x^0,\x^2,\ldots,\x^k\}$ is  the algorithm-generated filtration. We further assume that conditioned on $\bm{\cF}^{k}$, the random data  $\{\xi_i^t\}$ are independent of one another for any $\{i\}_{i=1}^n$ and $\{t \}_{t \leq k}$.   \qd
	\end{assumption}

	\section{Error Recursion} 
	To establish the convergence of \eqref{atc-gt}, we will first derive an error recursion that will be key to our enhanced bounds. Motivated by \cite{alghunaim2021unified}, the following result rewrites  algorithm \eqref{atc-gt} in an equivalent manner.
	\begin{lemma}[\sc \small Equivalent GT form]
		Let $\mathbf{x}^{0}$ take any arbitrary value and $\mathbf{z}^{0}=\zero$. Then for static graphs, the update for $\x^k$ in  algorithm \eqref{atc-gt} is equivalent to following updates for $k=1,2,\ldots$ \vspace{-1mm}
		\begin{subequations}   \label{alg_stochastic_UDA}
			\begin{align}
				\x^{k+1} &=   (2\W - \I)  \x^{k} -\alpha \W^2 \grad \F(\x^{k})  -  \mathbf{B} \z^{k}   \label{x_update}   \\
				\z^{k+1} &= \z^{k}+ \mathbf{B}  \x^{k} \label{z_update}
			\end{align}
		\end{subequations}
		with  initialization $\x^1=\W (\x^0 -\alpha \grad \F(\x^0))$ and $\z^{1} =  \mathbf{B}  \x^{0}$,	and  $\B=\I-\W$.
	\end{lemma}
	\begin{proof}
		Clearly with the above initialization, both $\x^1$ are identical for the updates \eqref{atc-gt} and \eqref{alg_stochastic_UDA}. Now, for $k \geq 1$, it holds from \eqref{x_update} that 
		\[
		\begin{aligned}
			\x^{k+1} - \x^{k} &=   (2\W - \I)  (\x^{k}- \x^{k-1})  -  \mathbf{B} (\z^{k}-\z^{k-1})  \\
			& \quad -\alpha \W^2 ( \grad \F(\x^{k})- \grad \F(\x^{k-1})).
		\end{aligned}
		\]
		Substituting $	\z^{k} - \z^{k-1}= \mathbf{B}  \x^{k-1}$ (\eqref{z_update}) and $\B=\I-\W$ into the above equation and rearranging the recursion gives 
		\[ 
		\begin{aligned}
			\x^{k+1}  	  &=  2\W  \x^{k}- \W^2 \x^{k-1} -\alpha \W^2 ( \grad \F(\x^{k})- \grad \F(\x^{k-1})) .
		\end{aligned}
		\]
		Following the same approach, we can also describe  the $\x^k$ update for the GT algorithm \eqref{atc-gt}  as above -- see \cite{alghunaim2019decentralized,alghunaim2021unified}. Hence, both methods are equivalent for static graph $\W$.
	\end{proof}

	Under Assumption \ref{assump:network}, the fixed point of recursion \eqref{alg_stochastic_UDA}, denoted by  $(\x^\star ,\z^\star)$,   satisfies:
	\begin{equation} \label{fixed_point}
		\begin{aligned}
			\zero &=  \alpha \W^2 \grad \mathbf{f}(\mathbf{x}^\star)   + \mathbf{B} \mathbf{z}^\star \  \\
			\zero &=  \mathbf{B}  \mathbf{x}^\star.
		\end{aligned}
	\end{equation}
	where $\x^\star=\one \otimes x^\star$ and $x^\star$ is the optimal solution of \eqref{min_consenus}. The existence of $\z^\star$ can be shown by using similar arguments as in \cite[Lemma 3.1]{shi2015extra} or \cite[Lemma 1]{alghunaim2019decentralized}. By introducing the notation
	\begin{align} \label{def_error}
		\tilde{\x}^{k} \define \x^k -\x^\star, \quad \tilde{\z} \define \z^k- \z^\star,
	\end{align}
	using \eqref{alg_stochastic_UDA} and the fact  $(2\W - \I)\x^\star=\x^\star$, we can get the error recursion:
	\begin{equation} \label{error_uda1}
		\begin{aligned} 
			\begin{bmatrix}
				\tilde{\x}^{k+1} \\
				\tilde{\z}^{k+1}
			\end{bmatrix}	&= \begin{bmatrix}
				2\W-\I & - \B \\
				\B & \I
			\end{bmatrix} \begin{bmatrix}
				\tilde{\x}^{k} \\
				\tilde{\z}^{k}
			\end{bmatrix}  \\
			& \qquad - \alpha \begin{bmatrix}
				\W^2  \big(\grad \mathbf{f}(\mathbf{x}^{k})-\grad \f(\x^\star)+ \v^k \big) \\
				0
			\end{bmatrix},
		\end{aligned}
	\end{equation}
	where $
	\v^k \define \grad \F(\x^k)-\grad \f (\x^k)$. 
	\begin{remark}[\sc \small Alternative analysis approach] \rm
		By describing GT \eqref{atc-gt} in the alternative form \eqref{alg_stochastic_UDA}, we are able to derive the error recursion from the {\em fixed point} \eqref{error_uda1}. This is similar to the way Exact-diffusion/D$^2$ is analyzed in \cite{yuan2020influence,yuan2021removing}. This alternative approach allows us to derive tighter bounds compared with existing GT works \cite{pu2021distributed,xin2021improved,koloskova2021improved,alghunaim2021unified}. \qd
	\end{remark}
	
	Convergence analysis of \eqref{error_uda1} still  remains difficult. We will exploit the properties of the matrix $\W$ to transform recursion \eqref{error_uda1} into a more suitable form for our analysis.  To that end,  the following quantities are introduced:
	\begin{subequations} \label{avg_def_quantites}
		\begin{align}
			\bar{x}^{k} &\define\frac{1}{n} (\one_n\tran \otimes I_d) \x^{k}=\frac{1}{n} \sum_{i=1}^n x_i^k  ,
			\\
			\bar{e}_x^k& \define \frac{1}{n}(\one_n\tran \otimes I_d) \tilde{\x}^{k}=  \bar{x}^{k} -x^\star ,
			\\
			\overline{\grad f}(\x^k)& \define \frac{1}{n} (\one_n\tran \otimes I_d) \grad \f(\x^k)=\frac{1}{n} \sum_{i=1}^n \grad f_i(x_i^k), \\
			\bar{v}^{k} &\define\frac{1}{n} (\one_n\tran \otimes I_d) \v^{k}.
		\end{align} 
	\end{subequations}
	Under Assumption \ref{assump:network}, the matrix $\W$ admits the following eigen-decomposition:
	\begin{align} \label{W_decompositon}
		\W= \U \mathbf{\Sigma} \U^{-1} = \underbrace{\begin{bmatrix}
				\one \otimes I_d & \hat{\U}
		\end{bmatrix}}_{\U} \underbrace{\begin{bmatrix}
				I_d & 0 \\
				0 & \mathbf{\Lambda}
		\end{bmatrix}}_{ \mathbf{\Sigma}} \underbrace{\begin{bmatrix}
				\frac{1}{n} \one\tran \otimes I_d \\ \hat{\U}\tran
		\end{bmatrix}}_{\U^{-1}} 
	\end{align}
	where  $\mathbf{\Lambda}$ is a diagonal matrix with eigenvalues strictly less than one and $\hat{\U}$ is an ${dn \times d(n-1)}$ matrix that satisfies
	\begin{subequations} \label{UUtran}
		\begin{align}
			\hat{\U}\tran \hat{\U}&=\I, \quad (\one\tran \otimes I_d) \hat{\U}=0 \\
			\hat{\U}\hat{\U}\tran&=\I-\tfrac{1}{n} \one \one\tran \otimes I_d.
		\end{align}
	\end{subequations}
	
	\begin{lemma}[\sc \small Decomposed error recursion] \label{lemma:error_decomposed}
		Under Assumption \ref{assump:network}, there exists matrices $\hat{\V}$ and $\mathbf{\Gamma}$  to transform the error recursion \eqref{error_uda1} into the following form:
		\begin{subequations} \label{error_diag_transformed}
			\begin{align}
				\bar{e}_x^{k+1} &=\bar{e}_x^{k} - \alpha  \overline{\grad f}(\x^k)+\alpha \bar{v}^k, \label{error_average_diag}  \\
				\hat{\x}^{k+1}&=\mathbf{\Gamma} \hat{\x}^{k} - \alpha  \hat{\V}_l^{-1}  
				\mathbf{\Lambda}^2 \hat{\U}\tran   \big(\grad \mathbf{f}(\mathbf{x}^{k})-\grad \f(\x^\star)+\v^k \big) , \label{error_hat_diag}  
			\end{align}
		\end{subequations}
		where 
		\begin{align} \label{x_hat_def}
			\hat{\x}^{k} \define \hat{\V}^{-1} \begin{bmatrix}
				\hat{\U}\tran \tilde{\x}^{k} \\
				\hat{\U}\tran \tilde{\z}^{k}
			\end{bmatrix},
		\end{align}
		and $ \hat{\V}_l^{-1}$ denotes the left block of $ \hat{\V}^{-1}=[ \hat{\V}_l^{-1} ~  \hat{\V}_r^{-1}]$.  Moreover, the following bounds hold:  
		\begin{align} \label{gt_diff_W_bound}
			\|\hat{\V}\|^2 &\leq 3, \quad
			\|\hat{\V}^{-1}\|^2 \leq 9, \quad   \|\mathbf{\Gamma}\| \leq  \tfrac{1+\lambda}{2},
		\end{align}
		where $\lambda=\max_{i \in \{2,\dots,n\}} \lambda_i$.
	\end{lemma}
	\begin{proof}
		See Appendix \ref{app:lemma_decomposition}
	\end{proof}
	The preceding result will serve as the starting point for deriving the bounds that will lead us to our conclusions. Specifically, we can derive the following bounds from the above result.
	\begin{lemma} [\sc \small Coupled error inequality] \label{lemma_coupled}
		Suppose Assumptions \ref{assump:network}--\ref{assump:smoothness} hold. Then, if $\alpha < \tfrac{1}{4 L}$, we have
		\begin{align} \label{ineq_average}
			\Ex	\|\bar{e}_x^{k+1}\|^2  
			& \leq (1- \mu \alpha) \Ex \|	\bar{e}_x^{k}\|^2  - \alpha  \big( \Ex f(\bar{x}^{k})-f(x^\star)\big) \nonumber \\
			& \quad  + \frac{3 \alpha c_1^2 L }{2n} \Ex \| 	\hat{\x}^{k} \|^2 + \frac{\alpha^2 \sigma^2}{n}, 
		\end{align}
		and
		\begin{align} \label{ineq_cons}
			\Ex  \|\hat{\x}^{k+1}\|^2  & \leq   \gamma \Ex \| \hat{\x}^{k}\|^2 + \frac{\alpha^2  c_2^2 \lambda^4}{ (1-\gamma)} \Ex  \|\grad \mathbf{f}(\mathbf{x}^{k})-\grad \f(\x^\star)  \|^2  \nonumber \\
			& \quad 
			+  \alpha^2  c_2^2  \lambda^4 n \sigma^2,
		\end{align}
		where $\gamma \define \|\mathbf{\Gamma}\|$, $c_1 \define \|\hat{\V}\|$, and $c_2=\|\hat{\V}^{-1}\|$.
	\end{lemma}
	\begin{proof}
		See Appendix \ref{app:lemma_coupled}.
	\end{proof}

	\section{Convergence Results}
	In this section, we present our main convergence results in Theorems \ref{thm_cvx_convergence} and \ref{thm_strong_cvx_convergence}. We then discuss our results and highlight the differences with existing bounds.

	\begin{theorem} [\sc \small Convex case] \label{thm_cvx_convergence}
		Suppose that Assumptions \ref{assump:network}-\ref{assump:smoothness} are satisfied. Then, there exists a constant stepsize $\alpha$ such that
		\begin{align} \label{cvx_theorem}
			&	\frac{1}{K}   \sum\limits_{k=0}^{K-1} 	\left(	\Ex [f(\bar{x}^{k})-f^\star] + \frac{ L }{ n}	  \Ex  	\|\x^{k} - \one \otimes \bar{x}^{k}\|^2 \right)  \nonumber \\
			&\leq \frac{  \sigma  \|	\bar{e}_x^{0}\|}{\sqrt{n K}}
			+ \left(\frac{ L \lambda^4 \sigma^2}{1-\lambda}\right)^{1 / 3}\left(\frac{ \|	\bar{e}_x^{0}\|^2}{K}\right)^{\tfrac{2}{3}}
			\nonumber \\
			& \quad + \left( \frac{ L \lambda^2}{1-\lambda} \|	\bar{e}_x^{0}\|^2 + \frac{  \varsigma_{\star}^2      }{ L (1-\lambda)  }   \right) \frac{C}{K},
		\end{align}
		where $	\bar{e}_x^0 \define  \bar{x}^{0} -x^\star $, $\varsigma_{\star}^2 \define \frac{1}{n} \sum_{i=1}^n \|\grad f_i(x^\star)\|^2$, and $C$ is an absolute constant.
	\end{theorem}
	\begin{proof}
		See Appendix \ref{app:thm_cvx_proof}.
	\end{proof}

	\begin{theorem} [\sc \small Strongly-convex case] \label{thm_strong_cvx_convergence}
		Suppose that Assumptions \ref{assump:network}-\ref{assump:smoothness} are satisfied.  Then, there exists a constant stepsize $\alpha$ such that
		\begin{align}  \label{scvx_theorem}
			&		\Ex  \|\bar{e}_x^{K}\|^2  +  \tfrac{1}{n}\|\x^{K} - \one \otimes \bar{x}^{K}\|^2  
			\leq \tilde{\mathcal{O}}\left(\frac{\sigma^2}{n  K}+\frac{\sigma^2 }{(1-\lambda) K^{2}}
			\right)
			\nonumber	\\
			& \quad + \tilde{\mathcal{O}}\left(  \frac{\sigma^2 }{(1-\lambda)^2 n  K^{3}}
			+ (a_0+\varsigma_{\star}^2)  \exp \left[-(1-\lambda) K \right] \right),
		\end{align}
		where $a_0 \define	 \|\bar{x}^{0} -x^\star\|^2 $, $\varsigma_{\star}^2 \define \frac{1}{n} \sum_{i=1}^n \|\grad f_i(x^\star)\|^2$, and the notation $\tilde{\mathcal{O}}(\cdot)$ ignores logarithmic factors.
	\end{theorem}
	\begin{proof}
		See Appendix \ref{app:thm_strong_cvx_proof}.
	\end{proof} 
	In comparison to \cite{koloskova2021improved}, our results removes the log factor $\cO(\log(\frac{1}{1-\lambda}) )$ and holds for any number of iteration $K$ -- see Table \ref{table}. Moreover, observe that for the strongly-convex case, unlike \cite{koloskova2021improved}, we do not have a network term $1/(1-\lambda)$ multiplying the highest order exponential term $\exp(\cdot)$. 
	
	\begin{remark}[\sc \small Improvement upon nonconvex GT rates] \rm
		The GT rates for convex and strongly-convex settings provided in Theorems \ref{thm_cvx_convergence} and \ref{thm_strong_cvx_convergence} improve upon the GT rates for non-convex \cite{alghunaim2021unified,koloskova2021improved}  and PL condition \cite{alghunaim2021unified} settings. For example, observe  from Table \ref{table} that the GT rate under the PL condition \cite{alghunaim2021unified} is $\frac{1}{n  \epsilon}+
		\left(\frac{1}{(1-\lambda)^{1/2}} + \frac{1}{(1-\lambda) \sqrt{n}} \right) \frac{1}{\sqrt{\epsilon}}
		+ \frac{1}{1-\lambda} \log \left(\frac{(a_0+\varsigma_{\star}^2)}{\epsilon}\right)$, which  has an additional term $\frac{1}{(1-\lambda) \sqrt{n}}  \frac{1}{\sqrt{\epsilon}}$ compared to our strongly-convex rate. \qd
	\end{remark}

	\begin{remark}[\sc \small Comparison with Exact-diffusion/D$^2$ \cite{yuan2021removing}] \rm
		For the convex case,  the difference with Exact-diffusion/D$^2$ \cite{yuan2021removing} is in the highest order term.  Exact-diffusion/D$^2$  is $\left(\frac{a_0}{(1-\lambda) } + \varsigma_\star^2 \right) \frac{1}{K}$ while GT is $\left(\frac{a_0}{(1-\lambda) } + \frac{\varsigma_\star^2}{(1-\lambda) } \right) \frac{1}{K}$ where GT has $1/(1-\lambda)$ multiplied by $\varsigma_\star^2$, which is slightly worse than Exact-diffusion/D$^2$. A similar conclusion can be reached for the strongly-convex scenario. \qd
	\end{remark}
	
	\section{Simulation results}
	This section will present several numerical simulations that compare Gradient-tracking with centralized SGD (CSGD) and decentralized SGD (DSGD). 
	
	\textbf{Linear regression.} We consider solving a strongly-convex problem \eqref{min_consenus} with $f_i(x) = \frac{1}{2}\mathbb{E}(a_i^\top x - b_i)^2$ in which random variable $a_i \sim \mathcal{N}(0, I_d)$, $b_i = a_i\tran x_i^\star + n_i$ for some local solution $x_i^\star \in \mathbb{R}^d$ and $n_i \sim \mathcal{N}(0, \sigma_n^2)$. The stochastic gradient is calculated as $\nabla F_i(x) = a_i(a\tran_i x - b_i)$. Each local solution $x_i^\star = x^\star + v_i$ is generated using the formula $x_i^\star = x^\star + v_i$, where $x^\star \sim \mathcal{N}(0, I_d)$ is a randomly generated global solution while $v_i \sim \mathcal{N}(0, \sigma_v^2 I_d)$ controls similarities between local solutions.
	
	Generally speaking, a large $\sigma_v^2$ will result in local solutions $\{x_i^\star\}_{i=1}^n$ that are vastly different from one another. We used $d=5$, $\sigma
	_n^2 = 0.01$, and $\sigma_v^2 = 1$ in simulations. Experiments are carried out on ring and exponential graphs of size $n=30$, respectively. Each algorithm's stepsize (learning rate) is carefully tuned so that they all converge to the same relative mean-square-error. Each simulation is run $30$ times, with the solid line representing average performance and the shadow representing standard deviation. The results are depicted in Fig. \ref{fig-linear-regression}. The relative error is shown on the $y$-axis as $\frac{1}{n} \sum_{i=1}^n \mathbb{E} \|x^k_i - x^\star\|^2/\|x^\star\|^2$. When running over the exponential graph which has a well-connected topology with $1 - \lambda = 0.33$, it is observed that both DSGD and Gradient-tracking perform similarly to CSGD. However, when running over the ring graph which has a badly-connected topology with $1 - \lambda = 0.0146$, DSGD gets far slower than CSGD due to its sensitivity to network topology. In contrast, Gradient-tracking just gets a little bit slower than CSGD and performs far better than DSGD. This phenomenon coincides with our established complexity bound in Table \ref{table_transient_time_PL_bound} showing that GT has a much weaker dependence on network topology (i.e., $1-\lambda$). 
	
	\begin{figure}[t]
		\centering 
		\includegraphics[width=4.3cm]{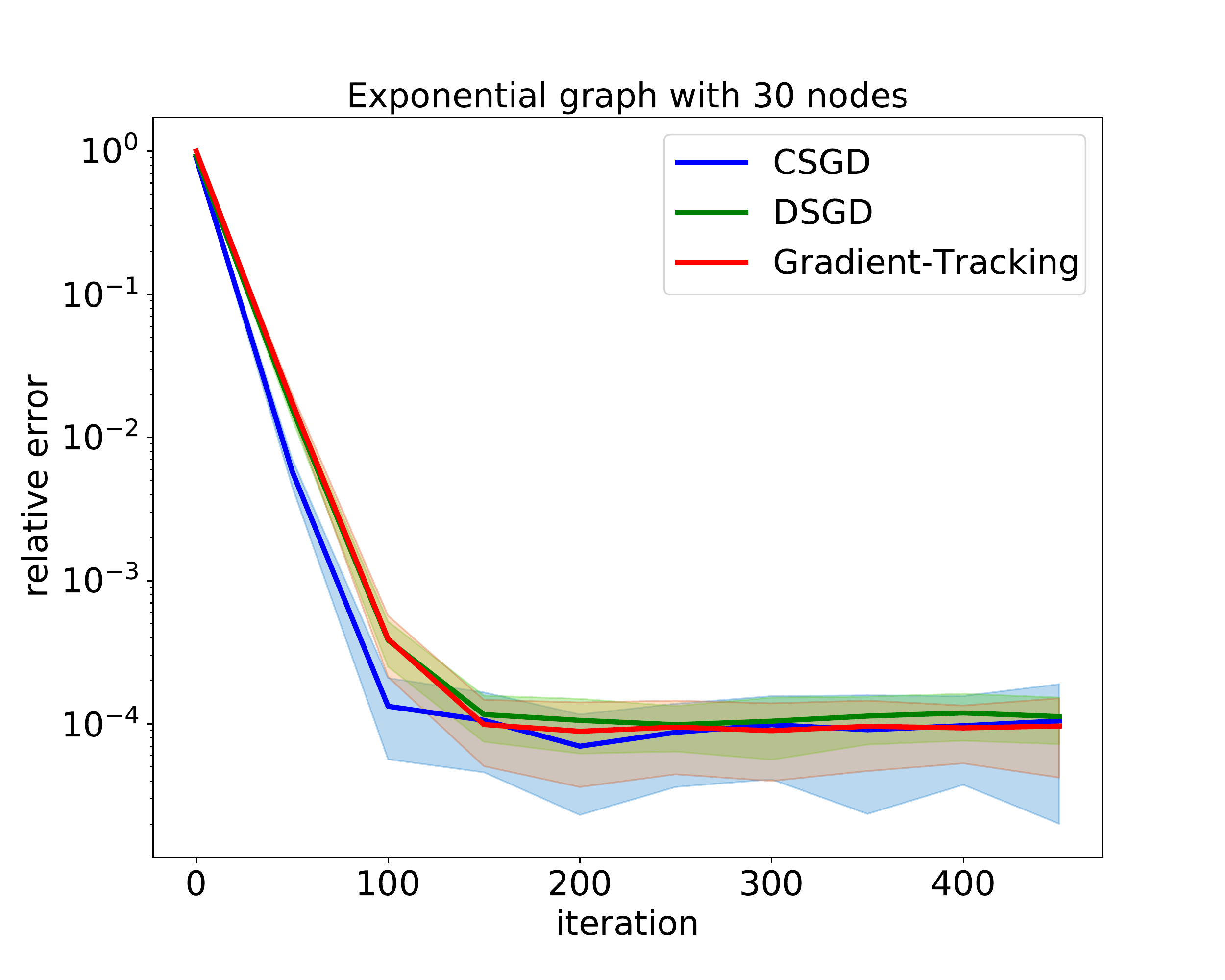}
		\includegraphics[width=4.3cm]{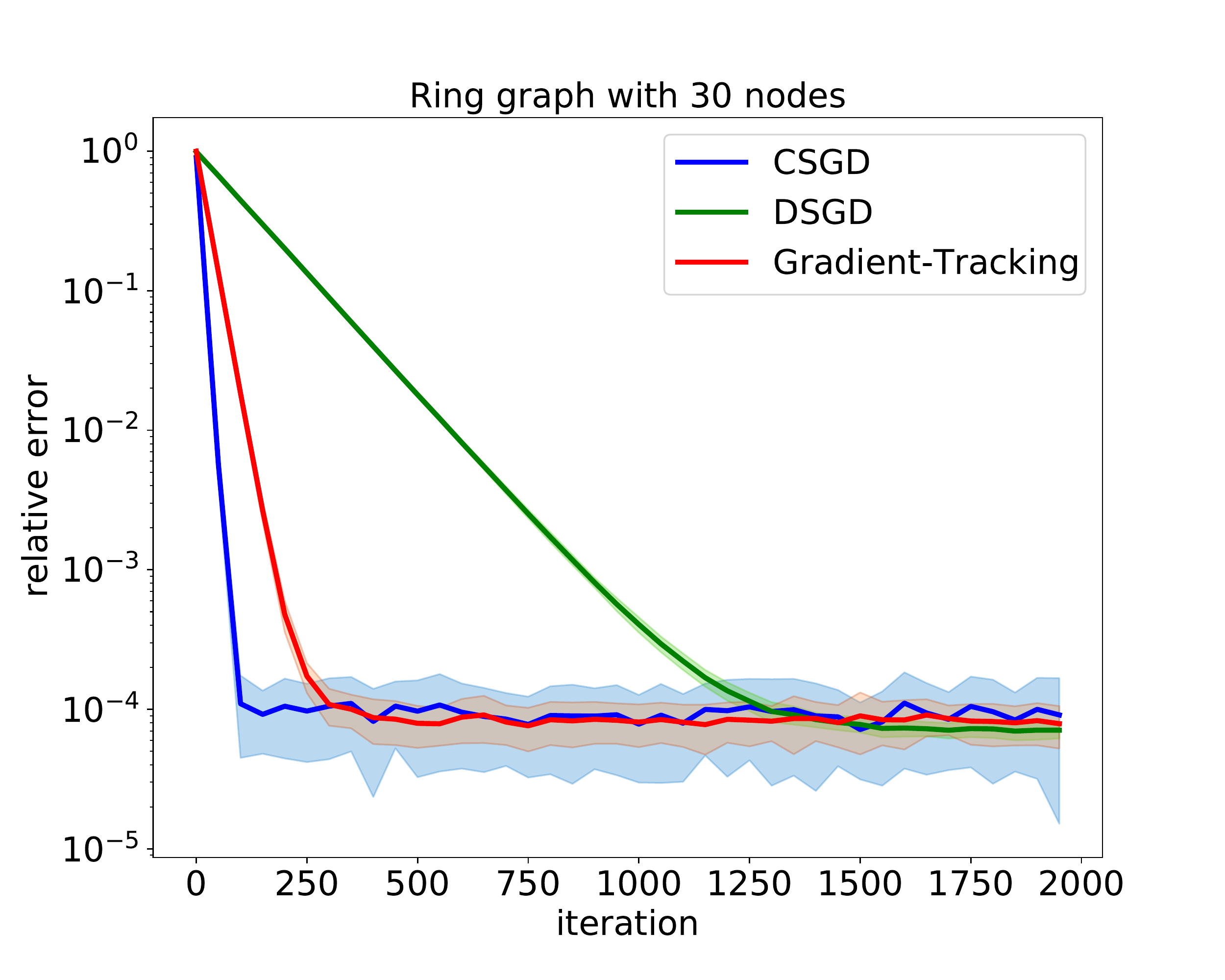}
		\caption{\small Comparison between different algorithms over exponential and ring graphs when solving distributed linear regression with heterogeneous data distributions. The spectral gap $1-\lambda$ is $0.33$ and $0.0146$ for exponential and ring graphs, respectively.}
		\label{fig-linear-regression}
	\end{figure}

	
	\textbf{Logistic regression.} We next consider the logistic regression problem, which has $f_i(x) = \mathbb{E} \ln(1+\exp(-y_i h_i\tran x)) $ where $(h_i, y_i)$ represents the training dataset stored in node $i$ with $h_i \in \mathbb{R}^d$ as the feature vector and $y_i \in-\{1,+1\}$ as the label. This is a convex but not strongly-convex problem. Similar to the linear regression experiments, we will first generate a local solution $x^\star_i$ based on $x^\star_i = x^\star + v_i$ using $v_i \sim \mathcal{N}(0, \sigma_v^2 I_d)$. We can generate local data that follows distinct distributions using $x_i^\star$. To this end, we generate each feature vector $h_i \sim \cN(0, I_d)$ at node $i$. To produce
	the corresponding label $y_{i}$, we create a random variable $z_{i} \sim \cU(0,1)$. If $z_{i} \le 1 + \exp(-y_{i}h_{i}\tran x_i^\star)$, we set $y_{i} = 1$; otherwise $y_{i} = -1$. Clearly, solution $x_i^\star$ controls the distribution of the labels. By adjusting $\sigma^2_v$, we can easily control data heterogeneity. The remaining parameters are the same as in linear regression experiments. The performances of each algorithm in logistic regression depicted in Fig.~\ref{fig-log-regression} are consistent with that in linear regression, i.e., Gradient-tracking performs well for both graphs while DSGD has a significantly deteriorated performance over the ring graph due to its less robustness to network topology.
	
	\begin{figure}[t]
		\centering 
		\includegraphics[width=4.3cm]{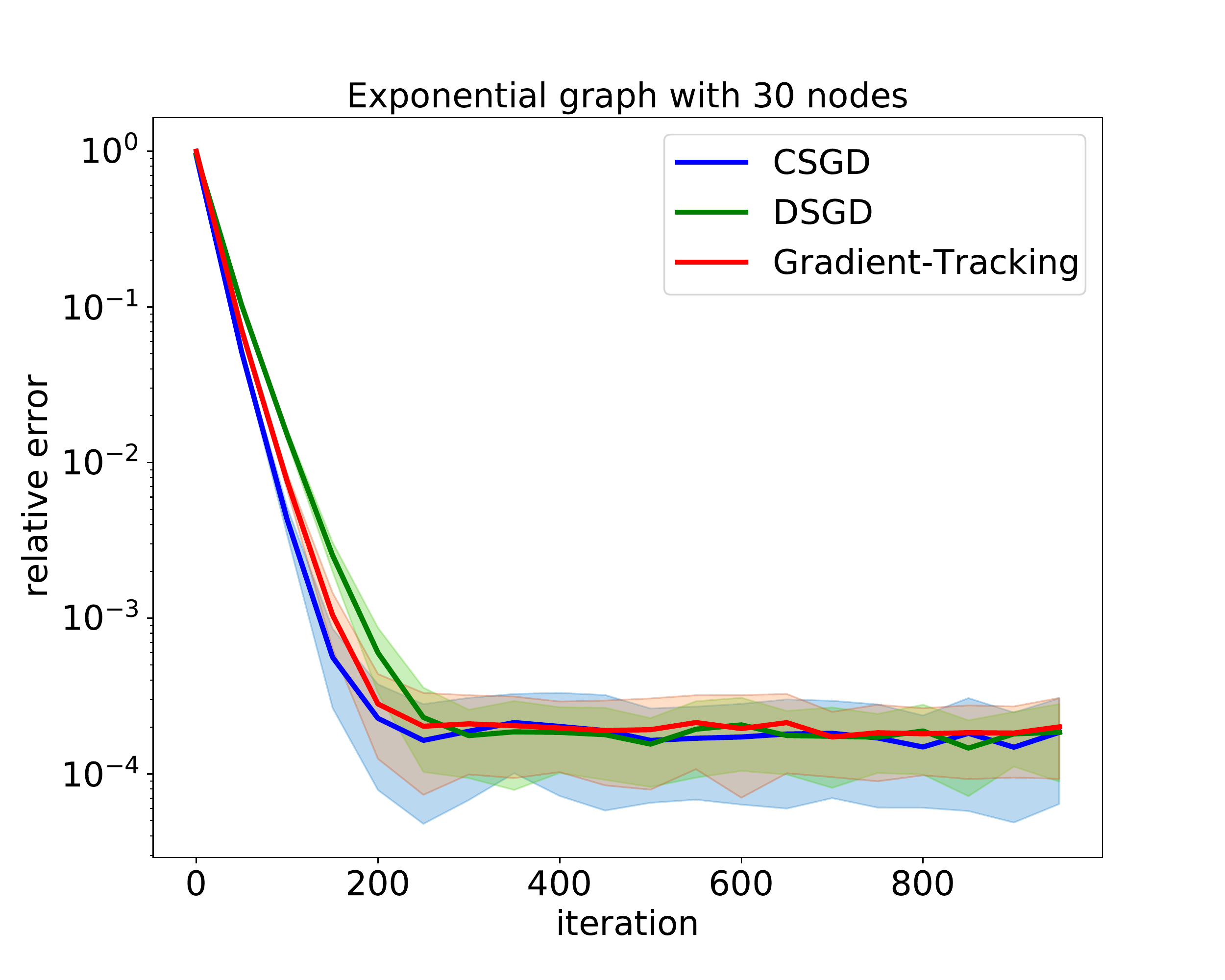}
		\includegraphics[width=4.3cm]{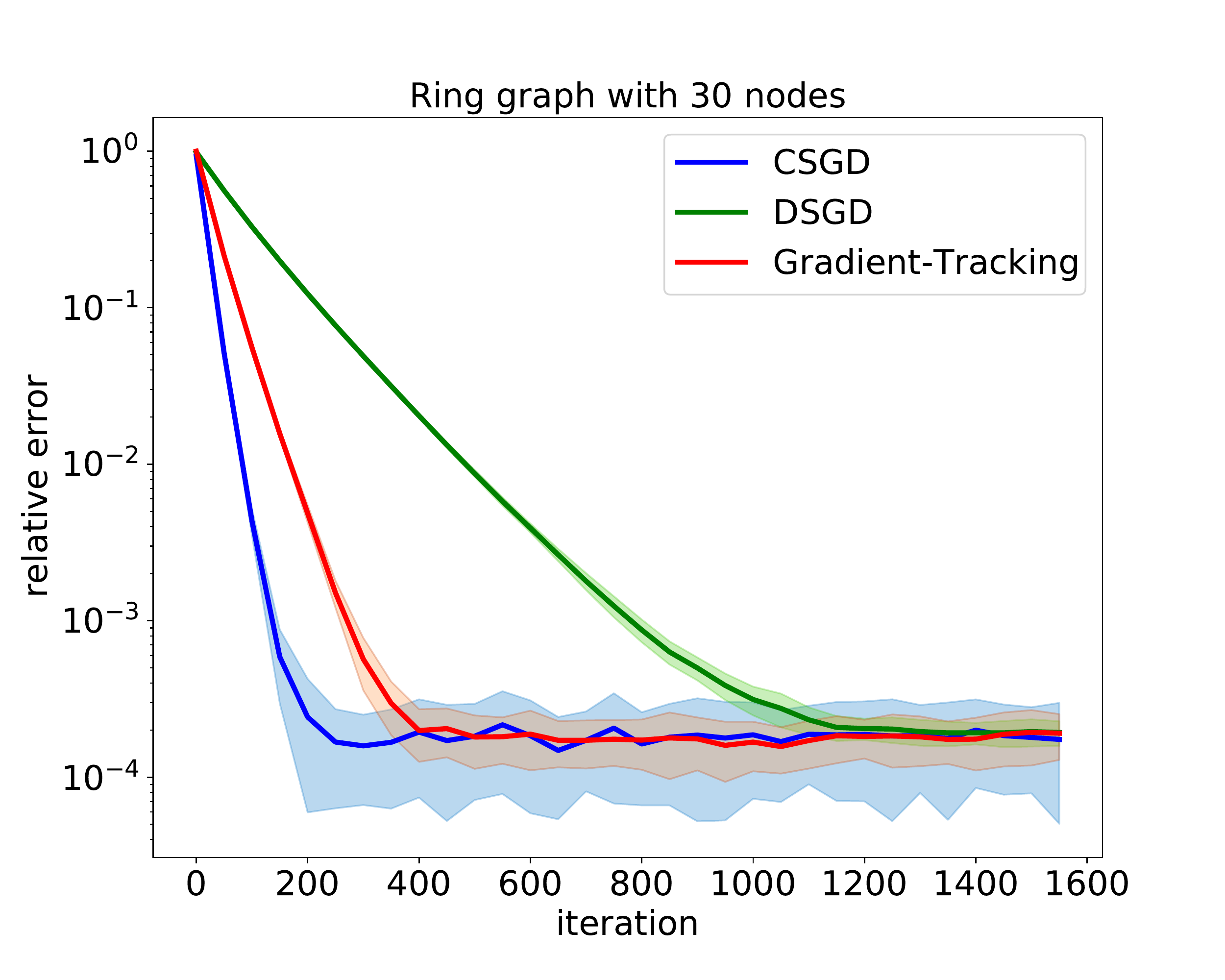}
		\caption{\small Comparison between different algorithms over exponential and ring graphs when solving distributed logistic regression.}
		\label{fig-log-regression}
	\end{figure}

	\appendices
	\section{Decomposed Error Recursion \\
		Prof of Lemma \ref{lemma:error_decomposed}} \label{app:lemma_decomposition}
	Using the decomposition \eqref{W_decompositon} and $\B=\I-\W$:
	\begin{subequations} \label{AB_decompositon}
		\begin{align}
			\W^2&=\U \mathbf{\Sigma}^2 \U^{-1} = \begin{bmatrix}
				\one \otimes I_d & \hat{\U}
			\end{bmatrix} \begin{bmatrix}
				I_d & 0 \\
				0 & \mathbf{\Lambda}^2
			\end{bmatrix} \begin{bmatrix}
				\frac{1}{n} \one\tran \otimes I_d \\ \hat{\U}\tran
			\end{bmatrix} \\
			\B&=\U (\I-\mathbf{\Sigma}) \U^{-1} = \begin{bmatrix}
				\one \otimes I_d & \hat{\U}
			\end{bmatrix} \begin{bmatrix}
				0 & 0 \\
				0 & \I-\mathbf{\Lambda}
			\end{bmatrix} \begin{bmatrix}
				\frac{1}{n} \one\tran \otimes I_d \\ \hat{\U}\tran
			\end{bmatrix},
		\end{align}
	\end{subequations}
	with   $\I-\mathbf{\Lambda}>0$.  	Substituting \eqref{AB_decompositon} into \eqref{error_uda1} and multiplying both sides by $\bdiag\{\U^{-1},\U^{-1}\}$ on the left, we obtain
	\begin{equation} \label{error_uda2}
		\begin{aligned} 
			\begin{bmatrix}
				\U^{-1} 	\tilde{\x}^{k+1} \\
				\U^{-1} 	\tilde{\z}^{k+1}
			\end{bmatrix}	&= \begin{bmatrix}
				2\mathbf{\Sigma}^2 -\I & - (\I- \mathbf{\Sigma}) \\
				\I- \mathbf{\Sigma} & \I
			\end{bmatrix} \begin{bmatrix}
				\U^{-1} 	\tilde{\x}^{k} \\
				\U^{-1} 	\tilde{\z}^{k}
			\end{bmatrix}  \\
			& \quad - \alpha \begin{bmatrix}
				\mathbf{\Sigma}^2 \U^{-1}  \big(\grad \mathbf{f}(\mathbf{x}^{k})-\grad \f(\x^\star)+\v^k \big) \\
				0
			\end{bmatrix}.
		\end{aligned}
	\end{equation}
	Since $\tilde{\z}^{k}$ always lies in the range space of $\B$, we have  $(\one_n\tran \otimes I_d) \tilde{\z}^{k}=0$ for all $k$. Using, the structure of $\U$ from \eqref{W_decompositon} and the definitions \eqref{avg_def_quantites}, we have
	\begin{align*}
		\U^{-1}  \tilde{\x}^{k}&=\begin{bmatrix}
			\bar{e}_x^k  \vspace{0.5mm} \\
			\hat{\U}\tran \tilde{\x}^{k}
		\end{bmatrix}, \quad \U^{-1}  \tilde{\z}^{k} =\begin{bmatrix}
			0  \vspace{0.5mm}  \\
			\hat{\U}\tran \tilde{\z}^{k}
		\end{bmatrix} 
		\\
		\U^{-1}  \grad \mathbf{f}(\mathbf{x})&=\begin{bmatrix}
			\overline{\grad f}(\x^k)  \vspace{0.5mm} \\
			\hat{\U}\tran \grad \mathbf{f}(\mathbf{x})
		\end{bmatrix}.
	\end{align*}  
	Thus, by using the structure of $\mathbf{\Sigma}^2$ and $\mathbf{\Sigma}_b^2$ given in \eqref{AB_decompositon}, we can rewrite \eqref{error_uda2} as
	\begin{subequations} \label{error_uda_tranformed_nondiag}
		\begin{align}
			\bar{e}_x^{k+1} &=\bar{e}_x^{k} - \alpha  \big(\overline{\grad f}(\x^k)-\overline{\grad f}(\x^\star) \big)  \\
			\begin{bmatrix}
				\hat{\U}\tran \tilde{\x}^{k+1}  \\
				\hat{\U}\tran \tilde{\z}^{k+1} 
			\end{bmatrix}&= \begin{bmatrix}
				2\mathbf{\Lambda} -\I & - (\I-\mathbf{\Lambda}) \\
				\I-\mathbf{\Lambda} & ~\I
			\end{bmatrix}  \begin{bmatrix}
				\hat{\U}\tran \tilde{\x}^{k}    \\
				\hat{\U}\tran \tilde{\z}^{k}
			\end{bmatrix}   \nonumber \\
			& \quad - \alpha \begin{bmatrix}
				\mathbf{\Lambda}^2 \hat{\U}\tran   \big(\grad \mathbf{f}(\mathbf{x}^{k})-\grad \f(\x^\star)\v^k \big) \\
				0
			\end{bmatrix}.
		\end{align}
	\end{subequations}
	Let
	\begin{align}
		\mathbf{G} \define \begin{bmatrix}
			2\mathbf{\Lambda} -\I & - (\I-\mathbf{\Lambda}) \\
			\I-\mathbf{\Lambda} & ~\I
		\end{bmatrix} .
	\end{align}
	It is important to note that the matrix $\G$ is identical to the one studied in \cite{alghunaim2021unified} (for nonconvex case). Therefore,   following  the same arguments used in  \cite[Appendix B]{alghunaim2021unified}, we can decompose it as  $\mathbf{G}  = \hat{\V} \mathbf{\Gamma} \hat{\V}^{-1}$
	for  matrices $\hat{\V}$ and $\mathbf{\Gamma}$ satisfying the conditions in the lemma. Multiplying the second equation  in \eqref{error_uda_tranformed_nondiag} by $ \hat{\V}^{-1}$, we arrive at \eqref{error_diag_transformed}.
	
	\section{Coupled Error Inequalities \\
		Proof of Lemma \ref{lemma_coupled}} 
	\label{app:lemma_coupled}
	\subsection*{\bf Proof of inequality \eqref{ineq_average}}
	The proof adjusts the argument from \cite[Lemma 8]{koloskova2020unified}. Using  \eqref{error_average_diag} and  Assumption \ref{assump:noise}, we have
	\begin{align}
		&\Ex	[\|\bar{e}_x^{k+1}\|^2| \bm{\cF}^{k}] \nonumber \\
		&= \|	\bar{e}_x^{k} - \tfrac{\alpha}{n} \textstyle \sum_{i=1}^{n}(\nabla f_i(x^{k}_i) - \nabla f_i(x^\star))\|^2 + \alpha^2 \Ex [\|	\bar{v}^k\|^2| \bm{\cF}^{k}] \nonumber \\
		& \leq \|	\bar{e}_x^{k} - \tfrac{\alpha}{n} \textstyle \sum_{i=1}^{n}(\nabla f_i(x^{k}_i) - \nabla f_i(x^\star))\|^2 
		+ \frac{\alpha^2 \sigma^2}{n} \nonumber \\
		&= \|	\bar{e}_x^{k} \|^2  + \alpha^2 \| \tfrac{1}{n}\textstyle \sum\limits_{i=1}^{n} (\nabla f_i(x^{k}_i) - \nabla f_i(x^\star) )\|^2 
		\nonumber \\
		& \quad - \textstyle  \frac{2 \alpha}{n}    \sum\limits_{i=1}^{n} 	\left\langle \nabla f_i(x^{k}_i), \bar{e}_x^{k} \right\rangle 	+ \frac{\alpha^2 \sigma^2}{n} , 
		\label{xnbsdh}
	\end{align}
	where we used $\sum_{i=1}^{n} \nabla f_i(x^\star)=0$. 	  The second term on the right can be bounded as follows:
	\begin{align}
		& \alpha^2 \| \tfrac{1}{n}\textstyle \sum\limits_{i=1}^{n} \big(\nabla f_i(x^{k}_i) -\nabla f_i(\bar{x}^{k}) +\nabla f_i(\bar{x}^{k}) - \nabla f_i(x^\star) \big)\|^2 \nonumber \\
		& \leq  2 \alpha^2 \| \tfrac{1}{n}\textstyle \sum\limits_{i=1}^{n} (\nabla f_i(x^{k}_i) -\nabla f_i(\bar{x}^{k})) \|^2 \nonumber \\
		& \quad + 2  \alpha^2 \|\tfrac{1}{n} \textstyle \sum\limits_{i=1}^{n} ( \nabla f_i(\bar{x}^{k})-\nabla f_i(x^\star) )\|^2 \nonumber \\
		& \leq  \tfrac{2 \alpha^2}{n} \textstyle \sum\limits_{i=1}^{n} \|  \nabla f_i(x^{k}_i) -\nabla f_i(\bar{x}^{k}) \|^2
		\\
		& \quad  + 2 \alpha^2 \|  \nabla f(\bar{x}^{k}) - \nabla f(x^\star)\|^2 \nonumber \\
		& \leq  \tfrac{2 \alpha^2 L^2}{n}  \| \x^{k} - \one \otimes \bar{x}^{k} \|^2 + 2 \alpha^2  \|  \nabla f(\bar{x}^{k}) -\nabla f(x^\star)\|^2 \nonumber \\
		& \leq  \tfrac{2 \alpha^2 L^2}{n}  \| \x^{k} - \one \otimes \bar{x}^{k} \|^2 + 4L \alpha^2   (   f(\bar{x}^{k})-f(x^\star) ), \label{b_ncross}
	\end{align}
	where  the first two inequalities follows from Jensen's inequality. The third inequality follows from the Lipschitz gradient assumption. In the last inequality, we used the $L$-smoothness property of the aggregate function \cite{nesterov2013introductory}:  
	\[
	\| \nabla f(\bar{x}^{k})-\nabla f(x^\star)\|^2 \leq 2L \big( f(\bar{x}^{k})-f(x^\star) \big).
	\]
	Note that   for $L$-smooth and $\mu$-strongly-convex function $f$, it holds that \cite{nesterov2013introductory}:
	\begin{subequations}
		\begin{align}
			f(x)-f(y)-\tfrac{L}{2} \|x-y\|^2 &\leq \langle \grad f(y) ,(x-y) \rangle \label{bound:L_smooth_function} \\
			f(x)-f(y)+\tfrac{\mu}{2} \|x-y\|^2 &\leq \langle \grad f(x) , (x-y) \rangle. \label{bound:mu_sc_function} 
		\end{align}
	\end{subequations}
	Using these inequalities, the cross term in \eqref{b_ncross} can be bounded by
	\begin{align}
		& -\tfrac{2 \alpha}{n} \textstyle \sum\limits_{i=1}^{n} 	  \langle \nabla f_i(x^{k}_i), \bar{e}_x^{k} \rangle \nonumber \\
		&= \tfrac{2 \alpha}{n}     \textstyle  	\sum\limits_{i=1}^{n} \big(- \langle \nabla f_i(x^{k}_i),   \bar{x}^{k}-x_i^{k} \rangle - \langle \nabla f_i(x^{k}_i), x_i^{k}-x^\star \rangle  \big) \nonumber \\
		& \leq  \tfrac{2 \alpha}{n}    \textstyle  \sum\limits_{i=1}^{n} 	\bigg( - f_i(\bar{x}^{k})+f_i(x_i^{k})+ \tfrac{L}{2} \|\bar{x}^{k}-x_i^{k} \|^2 \nonumber \\
		& \quad -\tfrac{\mu}{2} \|x_i^{k}-x^\star\|^2-f_i(x_i^{k})+f_i(x^\star ) \bigg) \nonumber \\
		& \leq -2 \alpha \big( f(\bar{x}^{k})-f(x^\star)\big) 
		\nonumber \\
		& \quad  + \tfrac{L \alpha}{n}    \textstyle  \sum\limits_{i=1}^{n} 	 \|\bar{x}^{k}-x_i^{k} \|^2 - \mu \alpha \|\bar{x}^{k}-x^\star\|^2  \nonumber \\
		& = -2 \alpha \big( f(\bar{x}^{k})-f(x^\star)\big)  + \tfrac{L \alpha}{n}    \| \x^{k} - \one \otimes \bar{x}^{k} \|^2 - \mu \alpha \|\bar{e}_x^{k}\|^2,  \label{b_cross}
	\end{align}
	where the last inequality holds due to $-\frac{1}{n} \sum_{i=1}^n \|x_i^{k}-x^\star\|^2 \leq -  \|\frac{1}{n} \sum_{i=1}^n (x_i^{k}-x^\star)\|^2$.	 Substituting \eqref{b_ncross} and \eqref{b_cross} into \eqref{xnbsdh} and taking expectation, we obtain:
	\begin{align}
		\Ex	\|\bar{e}_x^{k+1}\|^2 
		& \leq (1- \mu \alpha)  \Ex \|	\bar{e}_x^{k}\|^2  -2 \alpha (1-2L \alpha )\Ex  \big( f(\bar{x}^{k})-f(x^\star)\big)     \nonumber \\
		& \quad + \tfrac{\alpha L}{n} \left(1+ 2 \alpha L \right)  \Ex \| \x^{k} - \one \otimes \bar{x}^{k} \|^2  + \tfrac{\alpha^2 \sigma^2}{n} \nonumber \\ 
		& \leq (1- \mu \alpha) \Ex \|	\bar{e}_x^{k}\|^2  - \alpha  \big( \Ex f(\bar{x}^{k})-f(x^\star)\big) \nonumber \\
		& \quad  + \frac{3 L \alpha}{2n} \Ex \| \x^{k} - \one \otimes \bar{x}^{k} \|^2 + \frac{\alpha^2 \sigma^2}{n},
		\label{last_cons}
	\end{align}
	where the last step uses $\alpha \leq \tfrac{1}{4L}$.
	Using \eqref{UUtran}, we have $\|\hat{\U}\tran \tilde{\x}^{k}\|^2=\|\hat{\U}\tran \hat{\U} \tilde{\x}^{k}\|^2=\|\x^{k} - \one \otimes \bar{x}^{k}\|^2$. Hence,
	\begin{align} \label{avg_hat_bound}
		\|\x^{k} - \one \otimes \bar{x}^{k}\|^2  \overset{\eqref{x_hat_def}}{=} \| \hat{\V} \hat{\x}^{k}\|^2 -
		\|\hat{\U}\tran \tilde{\z}^{k}\|^2 \leq  \|\hat{\V}\|^2 \| \hat{\x}^{k}\|^2 .
	\end{align}
	Substituting the above into \eqref{last_cons} yields \eqref{ineq_average}.

	\subsection*{\bf Proof of inequality \eqref{ineq_cons}}
	From \eqref{error_hat_diag}, we have
	\begin{align*}
		& \Ex [ \|\hat{\x}^{k+1}\|^2 |  \bm{\cF}^{k} ] \\
		&= \Ex  \left\| \mathbf{\Gamma} \hat{\x}^{k} - \alpha \hat{\V}_l^{-1} 
		\mathbf{\Lambda}^2 \hat{\U}\tran   \big(\grad \mathbf{f}(\mathbf{x}^{k})-\grad \f(\x^\star)+\v^k \big) |  \bm{\cF}^{k} \right\|^2 \\
		\overset{\eqref{noise_bound_eq_mean}}&{=}   \left\| \mathbf{\Gamma} \hat{\x}^{k} - \alpha  \hat{\V}_l^{-1} 
		\mathbf{\Lambda}^2 \hat{\U}\tran   \big(\grad \mathbf{f}(\mathbf{x}^{k})-\grad \f(\x^\star)\big)  \right\|^2 \\
		& \quad + \alpha^2  \Ex  \left\|   \hat{\V}_l^{-1} \mathbf{\Lambda}^2 \hat{\U}\tran  \v^k \big |  \bm{\cF}^{k} \right\|^2 \\
		\overset{\eqref{noise_bound_eq_variance}}&{\leq}      \left\| \mathbf{\Gamma} \hat{\x}^{k} - \alpha  \hat{\V}_l^{-1} 
		\mathbf{\Lambda}^2 \hat{\U}\tran   \big(\grad \mathbf{f}(\mathbf{x}^{k})-\grad \f(\x^\star)\big)  \right\|^2 \\
		& \quad +  \alpha^2 \|\hat{\V}_l^{-1} \|^2 \|	\mathbf{\Lambda}^2\|^2 \|\hat{\U}\tran\|^2 n \sigma^2. 
	\end{align*}
	Now,	for any vectors $\a$ and $\b$, it holds from Jensen's inequality that $\|\a+\b\|^2 \leq \frac{1}{\theta} \|\a\|^2 + \frac{1}{1-\theta}\|\b\| $ for any $\theta \in (0,1)$. Utilizing this bound with $\theta=\gamma \define \|\mathbf{\Gamma}\|$ on the first term of the previous inequality, we get
	\begin{align*}
		& \Ex [ \|\hat{\x}^{k+1}\|^2 |  \bm{\cF}^{k} ] \\
		& \leq   \gamma \| \hat{\x}^{k}\|^2 + \tfrac{\alpha^2 \| \hat{\V}_l^{-1} \|^2 \|	\mathbf{\Lambda}^2\|^2 \|\hat{\U}\tran\|^2}{ (1-\gamma)} \|\grad \mathbf{f}(\mathbf{x}^{k})-\grad \f(\x^\star)  \|^2 \\
		& \quad +  \alpha^2 \|\hat{\V}_l^{-1} \|^2 \|	\mathbf{\Lambda}^2\|^2 \|\hat{\U}\tran\|^2 n \sigma^2.
	\end{align*}
	Taking expectation and using $\|\hat{\U}\tran\| \leq 1$, $\|\hat{\V}_l^{-1} \|^2 \leq \|\hat{\V}^{-1} \|^2$, and $\|	\mathbf{\Lambda}^2\|^2 \leq \lambda^4$ yield our result \eqref{ineq_cons}.

	\section{Proof of Theorem \ref{thm_cvx_convergence}}
	\label{app:thm_cvx_proof}
	Using similar argument to \eqref{b_ncross} and \eqref{avg_hat_bound}, it holds that
	\begin{align*}
		&\|\grad \mathbf{f}(\mathbf{x}^{k})-\grad \f(\x^\star)  \|^2 
		\\
		&\leq 2\| \grad \mathbf{f}(\one \otimes \bar{x}^{k})- \grad \mathbf{f}(\x^\star) \|^2  +2 \|\grad \mathbf{f}( \mathbf{x}^{k})-  \grad \mathbf{f}(\one \otimes \bar{x}^{k}) \|^2  \\
		& \leq 4 n L [f(\bar{x}^{k}) - f(x^\star)]  +  2 c_1^2 L^2  \| \hat{\x}^{k}\|^2.
	\end{align*}
	Plugging the above bound into \eqref{ineq_cons} gives
	\begin{align*} 
		\Ex  \|\hat{\x}^{k+1}\|^2 
		& \leq   \left(\gamma  +  \tfrac{2 \alpha^2 c_1^2  c_2^2  L^2  \lambda^4}{ (1-\gamma)} \right) \Ex \| \hat{\x}^{k}\|^2 
		\nonumber \\
		& \quad + \tfrac{4 \alpha^2 c_2^2  L    \lambda^4 n }{ (1-\gamma)} \Ex \tilde{f}(\bar{x}^k)     
		+  \alpha^2 c_2^2  \lambda^4 n \sigma^2  \nonumber \\
		& \leq   \bar{\gamma} \Ex \| \hat{\x}^{k}\|^2 
		+ \tfrac{4  \alpha^2 c_2^2  L    \lambda^4 n }{ (1-\gamma)} \Ex \tilde{f}(\bar{x}^k)   
		+  \alpha^2  c_2^2 \lambda^4 n \sigma^2 ,
	\end{align*}
	where  $\tilde{f}(\bar{x}^k)  \define f(\bar{x}^{k}) - f(x^\star) $, $\bar{\gamma} \define \frac{1+\gamma}{2}$, and the last inequiality holds when $\gamma  +  \frac{2 \alpha^2 c_1^2  c_2^2  L^2  \lambda^4}{ (1-\gamma)} \leq \frac{1+\gamma}{2}$, which is satisfied for
	\begin{align} \label{step_cvx1}
		\alpha \leq \frac{1-\lambda}{4 c_1 c_2  L  \lambda^2}.
	\end{align}
	Iterating the last recursion (for any $k=1,2,\dots$) gives 
	\begin{align} \label{cons_ineq_nonconvex_proof}
		\Ex \|\hat{\x}^{k}\|^2 
		&\leq  
		\bar{\gamma}^{k}  \|\hat{\x}^{0}\|^2 
		+ \tfrac{4 \alpha^2 c_2^2  L \lambda^4 n }{(1-\gamma)} \textstyle \sum\limits_{\ell=0}^{k-1} \bar{\gamma}^{k-1-\ell} \Ex \tilde{f}(\bar{x}^\ell)
		\nonumber \\
		&~  + \textstyle \sum\limits_{\ell=0}^{k-1} \bar{\gamma}^{k-1-\ell}  \left( \alpha^2 c_2^2  \lambda^4 n \sigma^2 \right)  \nonumber \\
		&\leq  
		\bar{\gamma}^{k}  \|\hat{\x}^{0}\|^2  
		+    \tfrac{4 \alpha^2 c_2^2   L    \lambda^4 n }{ (1-\gamma)} \textstyle \sum\limits_{\ell=0}^{k-1} \bar{\gamma}^{k-1-\ell} \Ex \tilde{f}(\bar{x}^\ell)
		\nonumber \\
		& \quad + \tfrac{ \alpha^2 c_2^2  \lambda^4 n \sigma^2}{1-\bar{\gamma}}. 
	\end{align}  
	In the last inequality we used $\textstyle \sum_{\ell=0}^{k-1} \bar{\gamma}^{k-1-\ell} \leq \frac{1}{1-\bar{\gamma}}$. Averaging over $k=1,2\dots,K$ and using $\bar{\gamma} = \frac{1+\gamma}{2}$, it holds that
	\begin{align} 
		& \textstyle	\frac{1}{K}  \sum\limits_{k=1}^K	\Ex \|\hat{\x}^{k}\|^2 \nonumber \\ 
		& \leq	\tfrac{2 \|\hat{\x}^{0}\|^2}{(1-\gamma)K} 
		+  \tfrac{4 \alpha^2 c_2^2  L    \lambda^4 n }{ (1-\gamma) K} \textstyle \sum\limits_{k=1}^K  \sum\limits_{\ell=0}^{k-1} \left( \tfrac{1+\gamma}{2} \right)^{k-1-\ell} \Ex \tilde{f}(\bar{x}^\ell) 
		+ \tfrac{2 \alpha^2 c_2^2 \lambda^4 n \sigma^2}{1-\gamma} 
		\nonumber \\
		& \leq  
		\tfrac{2 \|\hat{\x}^{0}\|^2}{(1-\gamma)K} 
		+   \tfrac{8  \alpha^2 c_2^2  L    \lambda^4 n }{ (1-\gamma)^2 K} \textstyle \sum\limits_{k=0}^{K-1}  \Ex \tilde{f}(\bar{x}^k) 
		+ \tfrac{ 2 \alpha^2 c_2^2 \lambda^4 n \sigma^2}{1-\gamma}.
	\end{align}  
	It follows that
	\begin{align}  \label{bound_nonconv_cons_final}
		\textstyle	\frac{1}{K}  \sum\limits_{k=0}^{K-1}	\Ex \|\hat{\x}^{k}\|^2
		& \leq  
		\frac{3 \|\hat{\x}^{0}\|^2}{(1-\gamma)K} 
		+   \tfrac{8 \alpha^2 c_2^2  L    \lambda^4 n }{ (1-\gamma)^2 K} \textstyle \sum\limits_{k=0}^{K-1}  \Ex \tilde{f}(\bar{x}^k) 
		\nonumber \\
		& \quad 
		+ \frac{ 2 \alpha^2 c_2^2  \lambda^4 n \sigma^2}{1-\gamma}.
	\end{align}  
	where we added $\frac{\|\hat{\x}^{0}\|^2}{(1-\gamma)K}$ and used $\frac{\|\hat{\x}^{0}\|^2}{K} \leq \frac{\|\hat{\x}^{0}\|^2}{(1-\gamma)K}$. 
	Now when $\mu=0$, we can rearrange \eqref{ineq_average} to get
	\begin{align} 
		\Ex (f(\bar{x}^{k})-f(x^\star)) 	\
		& \leq \frac{1}{\alpha} \left( \Ex \|	\bar{e}_x^{k}\|^2 - \Ex	\|\bar{e}_x^{k+1}\|^2   \right)   \nonumber \\
		& \quad  + \frac{3 c_1^2 L }{2n} \Ex \| 	\hat{\x}^{k} \|^2 + \frac{\alpha \sigma^2}{n}.  
	\end{align}
	Averaging over $k=0,\ldots,K-1$ ($K \geq 1$), it holds that
	\begin{align} \label{sum_grad_ineq_nonconvex_proof}
		& \textstyle	\frac{1}{K}	 \sum\limits_{k=0}^{K-1} \Ex \tilde{f}(\bar{x}^{k})
		\leq \frac{\|	\bar{e}_x^{0}\|^2}{\alpha K}	  
		+ \frac{3 c_1^2 L }{2n K}	 \sum\limits_{k=0}^{K-1}   \Ex  \|\hat{\x}^k\|^2 
		+\frac{\alpha \sigma^2}{n}.  
	\end{align}		
	Multiplying inequality \eqref{bound_nonconv_cons_final} by $2 \times \frac{3 c_1^2 L }{2n}$, adding to \eqref{sum_grad_ineq_nonconvex_proof}, and rearranging we obtain 
	\begin{align}  \label{appppppp}
		&\left(1-\tfrac{24  \alpha^2 c_1^2 c_2^2 L^2 \lambda^4  }{(1-\gamma)^2}  \right)	\frac{1}{K} 	\textstyle  \sum\limits_{k=0}^{K-1} \Ex \tilde{f}(\bar{x}^{k}) + \frac{3 c_1^2 L }{2n K}	 \sum\limits_{k=0}^{K-1}   \Ex  \|\hat{\x}^k\|^2   \nonumber \\
		&  ~\leq \frac{\|	\bar{e}_x^{0}\|^2}{\alpha K}	   
		+ 	\frac{9  c_1^2  L   \|\hat{\x}^{0}\|^2}{(1-\gamma) n K} 
		+\frac{\alpha \sigma^2}{n}
		+  \frac{6 \alpha^2 c_1^2 c_2^2  L  \lambda^4  \sigma^2}{1-\gamma} .
	\end{align}	
	Notice from \eqref{x_hat_def} that
	\begin{align} 
		\|	\hat{\x}^{0} \|^2 \leq  \|\hat{\V}^{-1}\|^2 \left( \|\hat{\U}\tran \tilde{\x}^{0}\|^2 + \|	\hat{\U}\tran \tilde{\z}^{0} \|^2 \right).
	\end{align}
	If we start from consensual initialization $\x^0=\one \otimes x^0$ and use the fact $\z^0=0$, the above reduces to  
	\begin{align} \label{hat_0_bound}
		\|	\hat{\x}^{0} \|^2 \leq  \|\hat{\V}^{-1}\|^2 \|	\hat{\U}\tran \z^{\star} \|^2 
		\leq \frac{\alpha^2 c_2^2  \lambda^4}{(1-\lambda)^2} \|\hat{\U}\tran \grad \mathbf{f}(\mathbf{x}^\star) \|^2,
	\end{align}
	where the last step holds by using \eqref{fixed_point} and \eqref{AB_decompositon}, which implies that 
	$
	\hat{\U}\tran	 \mathbf{z}^\star  =  \alpha (\I-\mathbf{\Lambda})^{-1} \mathbf{\Lambda}^2 \hat{\U}\tran  \grad \mathbf{f}(\mathbf{x}^\star) 
	$.  Plugging the previous inequality into \eqref{appppppp} and  setting $	\frac{1}{2} \leq 1-\frac{24 \alpha^2 c_1^2 c_2^2  L^2 \lambda^4  }{(1-\gamma)^2} $, \ie,
	\begin{align} \label{step_size_nonconv_prof_2}
		\alpha \leq \frac{1-\lambda}{4 \sqrt{6}  c_1 c_2 L \lambda^2}, 
	\end{align}	
	gives
	\begin{align}  \label{Psi_k}
		\frac{1}{K}   \sum\limits_{k=0}^{K-1}  \cE_k
		&  \leq \underbrace{\frac{ \|	\bar{e}_x^{0}\|^2}{\alpha K}	   
			+a_1 \alpha + a_2 \alpha^2
		}_{\define \Psi_K}
		+ 	\frac{a^{\star} \alpha^2 }{  K},
	\end{align}	
	where we defined $\cE_k \define	\frac{1}{2}	\Ex \tilde{f}(\bar{x}^{k}) + \frac{3 c_1^2 L }{2 n}	  \Ex  \|\hat{\x}^k\|^2  $ and
	\begin{subequations} \label{a012_constants}
		\begin{align} 
			a^{\star} &\define  \frac{18  c_1^2 c_2^2   L \lambda^4    \|\hat{\U}\tran  \grad \mathbf{f}(\mathbf{x}^\star)\|^2}{(1-\lambda)^3 n } \\
			a_1 & \define  \frac{  \sigma^2}{n} \quad a_2 \define \frac{12 c_1^2 c_2^2  L  \lambda^4  \sigma^2}{1-\lambda}.
		\end{align}
	\end{subequations}
	We now select the stepsize $\alpha$ to arrive at our result in a manner similar to \cite{koloskova2020unified}. First note that the previous inequality holds for
	\begin{align} \label{alpha_underline}
		\alpha \leq \frac{1}{\underline{\alpha}} \define  \min \left\{\frac{1}{4L }, \frac{1-\lambda}{4 \sqrt{6}  c_1 c_2 L \lambda^2}\right\}.
	\end{align}
	Setting $\alpha=\min \left\{\left(\frac{ \|	\bar{e}_x^{0}\|^2}{a_1K}\right)^{\frac{1}{2}},\left(\frac{ \|	\bar{e}_x^{0}\|^2}{a_2K}\right)^{\frac{1}{3}}, \frac{1}{\underline{\alpha}} \right\} \leq \frac{1}{\underline{\alpha}}$ we have three cases: i) If  $\alpha=\frac{1}{\underline{\alpha}}$, which is smaller than both $\left(\frac{ \|	\bar{e}_x^{0}\|^2}{a_1 K}\right)^{\frac{1}{2}}$ and $\left(\frac{ \|	\bar{e}_x^{0}\|^2}{a_2 K}\right)^{\frac{1}{3}}$, then
	\begin{align*}
		\Psi_K &= \frac{ \underline{\alpha}  \|	\bar{e}_x^{0}\|^2}{ K}+ \frac{a_1}{\underline{\alpha}}+ \frac{a_2}{\underline{\alpha}^2} \\
		& \leq
		\frac{\underline{\alpha}  \|	\bar{e}_x^{0}\|^2}{ K}
		+ \left(\frac{ a_1  \|	\bar{e}_x^{0}\|^2}{K}\right)^{\frac{1}{2}}+a_2^{\frac{1}{3}}\left(\frac{ \|	\bar{e}_x^{0}\|^2}{K}\right)^{\frac{2}{3}};
	\end{align*}
	ii)	If $\alpha=\left(\frac{ \|	\bar{e}_x^{0}\|^2}{a_1K}\right)^{\frac{1}{2}} < \left(\frac{ \|	\bar{e}_x^{0}\|^2}{ a_2 K}\right)^{\frac{1}{3}}$, then
	\begin{align*}
		\Psi_K &\leq 2\left(\frac{ a_1  \|	\bar{e}_x^{0}\|^2 }{K}\right)^{\frac{1}{2}}+a_2 \left(\frac{ \|	\bar{e}_x^{0}\|^2}{a_1 K}\right) \\
		& \leq 2\left(\frac{ a_1  \|	\bar{e}_x^{0}\|^2 }{K}\right)^{\frac{1}{2}}+a_2^{\frac{1}{3}}\left(\frac{ \|	\bar{e}_x^{0}\|^2}{K}\right)^{\frac{2}{3}};
	\end{align*}
	iii)  If $\alpha=\left(\frac{ \|	\bar{e}_x^{0}\|^2}{a_2 K}\right)^{\frac{1}{3}}<\left(\frac{ \|	\bar{e}_x^{0}\|^2}{a_1K}\right)^{\frac{1}{2}}$, then
	\begin{align*}
		\Psi_K &\leq 2 a_2^{\frac{1}{3}}\left(\frac{ \|	\bar{e}_x^{0}\|^2}{K}\right)^{\frac{2}{3}}+a_1 \left(\frac{ \|	\bar{e}_x^{0}\|^2}{a_2 K}\right)^{\frac{1}{3}} \\
		&\leq 2 a_2^{\frac{1}{3}}\left(\frac{ \|	\bar{e}_x^{0}\|^2}{K}\right)^{\frac{2}{3}}+\left(\frac{ a_1  \|	\bar{e}_x^{0}\|^2}{K}\right)^{\frac{1}{2}}.
	\end{align*}
	Combining the above cases, we have
	\begin{align*}
		\Psi_K \leq 2\left(\frac{ a_1  \|	\bar{e}_x^{0}\|^2}{K}\right)^{\frac{1}{2}}+2 a_2^{1 / 3}\left(\frac{ \|	\bar{e}_x^{0}\|^2}{K}\right)^{\frac{2}{3}}+\frac{ \underline{\alpha}  \|	\bar{e}_x^{0}\|^2}{K}.
	\end{align*}
	Therefore, substituting into \eqref{Psi_k} we conclude that
	\begin{align*} 
		\frac{1}{K}   \sum\limits_{k=0}^{K-1} \cE_k
		& \leq 2\left(\tfrac{ a_1  \|	\bar{e}_x^{0}\|^2}{K}\right)^{\tfrac{1}{2}} 
		+2 a_2^{\frac{1}{ 3}}\left(\tfrac{ \|	\bar{e}_x^{0}\|^2}{K}\right)^{\tfrac{2}{3}}
		\\
		& \quad +\frac{( \underline{\alpha} \|	\bar{e}_x^{0}\|^2 + \tfrac{a^{\star}}{\underline{\alpha}^2} )}{K}.
	\end{align*}
	Plugging the constants \eqref{a012_constants} and the upper bound for $\underline{\alpha}$ in \eqref{alpha_underline}, and using $ \varsigma_{\star}^2 = \frac{1}{n} \|\hat{\U}\tran\grad \mathbf{f}(\mathbf{x}^\star)\|^2 = \frac{1}{n} \sum_{i=1}^n \|\grad f_i(x^\star)-\grad f(x^\star)\|^2$ yields our rate \eqref{cvx_theorem}.

	\section{Proof of Theorem \ref{thm_strong_cvx_convergence}}
	\label{app:thm_strong_cvx_proof} 
	Substituting the bound 
	\begin{align*}
		&	\|\grad \mathbf{f}(\mathbf{x}^{k})-\grad \f(\x^\star)  \|^2 \leq L^2 \|\mathbf{x}^{k}- \x^\star \|^2  \\ 
		&\leq 2 L^2 \|\mathbf{x}^{k}-  \one \otimes \bar{x}^{k} \|^2 
		+ 2 L^2 \| \one \otimes \bar{x}^{k}- \x^\star \|^2   \\
		& \leq   2 L^2  c_1^2 \| \hat{\x}^{k}\|^2
		+ 2 n L^2 \|	\bar{e}_x^{k}\|^2 ,
	\end{align*}
	into \eqref{ineq_cons}, we get
	\begin{align}
		&	\Ex  \|\hat{\x}^{k+1}\|^2 \nonumber \\ 
		& \leq   \left(\gamma + \tfrac{2 \alpha^2 c_1^2 c_2^2  L^2  \lambda^4}{ (1-\gamma)}\right) \Ex \| \hat{\x}^{k}\|^2 
		+ \tfrac{2 \alpha^2 c_2^2  L^2   \lambda^4 n}{  (1-\gamma)} \|\bar{e}_x^{k}\|^2
		+   \alpha^2 c_2^2 \lambda^4 n \sigma^2 \nonumber \\
		& \leq   \left(\frac{1+\gamma}{2} \right) \Ex \| \hat{\x}^{k}\|^2 + \tfrac{2 \alpha^2 c_2^2  L^2   \lambda^4 n}{  (1-\gamma)} \|\bar{e}_x^{k}\|^2
		+   \alpha^2 c_2^2 \lambda^4 n \sigma^2,
	\end{align}
	where we used condition \eqref{step_cvx1} in the last inequality. Using $- \alpha  \big( \Ex f(\bar{x}^{k})-f(x^\star)\big) \leq 0$ in \eqref{ineq_average} and combining with above, it holds that
	\begin{align} \label{linear_dynamical_error2}
		\begin{bmatrix} 
			\Ex  \|\bar{e}_x^{k+1}\|^2 \\
			\frac{  c_1^2 }{n}	\Ex \|\hat{\x}^{k+1}\|^2   
		\end{bmatrix}
		&\leq 
		\underbrace{\begin{bmatrix}
				1- \mu  \alpha  \vspace{0.5mm} 		 
				&
				\frac{3}{2} \alpha L      \vspace{0.5mm} \\
				\frac{2 \alpha^2 c_1^2 c_2^2 L^2   \lambda^4  }{ (1-\gamma)} 
				& 
				\frac{1+\gamma}{2}
		\end{bmatrix}}_{\define A}
		\begin{bmatrix} 
			\Ex  \|\bar{e}_x^k\|^2 \\
			\frac{  c_1^2 }{n} \Ex \|\hat{\x}^{k}\|^2  
		\end{bmatrix} 
		\nonumber \\
		& \quad 
		+ \underbrace{\begin{bmatrix}
				\frac{\alpha^2  \sigma^2 }{n} \\
				\alpha^2  c_1^2	c_2^2  \lambda^4  \sigma^2 
		\end{bmatrix}}_{\define b}.
	\end{align} 
	The spectral radius of the matrix $A$ can be upper bounded by:
	\begin{align} \label{rho_H}
		\rho(A) \leq \|A\|_1 &= \max \left\{
		1- \mu  \alpha
		+ 		\tfrac{2 c_1^2 c_2^2 \alpha^2 L^2 \lambda^4  }{ (1-\gamma)} 
		, ~
		\tfrac{1+\gamma}{2}+ 	\tfrac{3}{2}   L \alpha   
		\right\} \nonumber \\
		&\leq 1-\frac{\mu  \alpha}{2},
	\end{align}
	where the  last inequality holds under the  stepsize condition:
	\begin{align} \label{stepsize_strongcvx_proof_2}
		\alpha \leq \min\left\{\frac{\mu (1-\gamma)}{4 c_1^2 c_2^2 L^2 \lambda^4 }, \frac{1-\gamma}{3  L + \mu}\right\}.
	\end{align}
	Since $\rho(A) <1$, we can iterate inequality \eqref{linear_dynamical_error2} to get
	\begin{align} 
		\begin{bmatrix} 
			\Ex  \|\bar{e}_x^{k}\|^2 \\
			\frac{c_1^2}{n}	\Ex \|\hat{\x}^{k}\|^2  
		\end{bmatrix}
		& \leq  
		A^k
		\begin{bmatrix} 
			\Ex  \|\bar{e}_x^{0}\|^2 \\
			\frac{c_1^2}{n} \Ex \|\hat{\x}^{0}\|^2  
		\end{bmatrix}
		+ \sum_{\ell=0}^{k-1} A^\ell b \nonumber \\ 
		& \leq  
		A^k
		\begin{bmatrix} 
			\Ex  \|\bar{e}_x^0\|^2 \\
			\frac{c_1^2}{n} \Ex \|\hat{\x}^{0}\|^2  
		\end{bmatrix} 
		+(I- A)^{-1} b.
	\end{align}
	Taking the (induced) $1$-norm, using the sub-multiplicative  properties of matrix induced norms, it holds that
	\begin{align} \label{tran_SC_0}
		\Ex  \|\bar{e}_x^{k}\|^2  +  \tfrac{c_1^2}{n}	\Ex \|\hat{\x}^{k}\|^2  
		& \leq  
		\|A^k\|_1  \tilde{a}_0
		+ \left\|(I- A)^{-1} b \right\|_1
		\nonumber \\
		& \leq  
		\|A\|^k_1 \tilde{a}_0
		+ \left\|(I- A)^{-1} b \right\|_1.
	\end{align}
	where $\tilde{a}_0 =\Ex  \|\bar{x}^{0}-x^\star\|^2 + \frac{c_1^2}{n} \Ex \|\hat{\x}^{0}\|^2 $.  We now bound the last term by  noting that  
	\begin{align*}
		&(I-A)^{-1} b \\
		&=
		\tfrac{1}{\det(I-A)}
		\begin{bmatrix}
			\frac{1-\gamma}{2} \vspace{0.5mm} 		 
			&
			\frac{3}{2}   \alpha L      \vspace{0.5mm} \\
			\frac{2 \alpha^2 c_1^2 c_2^2 L^2  \lambda^4   }{ (1-\gamma)}
			& 
			\mu  \alpha
		\end{bmatrix} b \\
		& = 
		\frac{1}{\alpha  \mu (1-\gamma) (\frac{1}{2} - \frac{3 \alpha^2   c_1^2 c_2^2  L^3 \lambda^4 }{(1-\gamma)^2 \mu } ) }
		\begin{bmatrix}
			\frac{1-\gamma}{2} \vspace{0.5mm} 		 
			&
			\frac{3}{2}   \alpha L      \vspace{0.5mm} \\
			\frac{2 \alpha^2 c_1^2 c_2^2 L^2  \lambda^4   }{ (1-\gamma)}
			& 
			\mu  \alpha
		\end{bmatrix}
		\begin{bmatrix}
			\frac{\alpha^2  \sigma^2 }{n} \\
			\alpha^2 	c_1^2 c_2^2  \lambda^4   \sigma^2 
		\end{bmatrix}
		\nonumber \\
		& \leq  \frac{4}{ \alpha  \mu (1-\gamma)}  
		\begin{bmatrix}
			\frac{(1-\gamma) \alpha^2 \sigma^2}{2n} + \frac{3}{2} c_1^2 c_2^2 \alpha^3 L \lambda^4 \sigma^2  \vspace{2mm}
			\\
			\frac{2 \alpha^4 c_1^2 c_2^2  L^2 \lambda^4 \sigma^2}{n (1-\gamma)} + \alpha^3 c_1^2 c_2^2  \mu \lambda^4 \sigma^2
		\end{bmatrix},
	\end{align*}
	where $\det(\cdot)$ denotes the determinant operation. In the last step we used $\frac{1}{2} - \frac{3  c_1^2 c_2^2 \alpha^2  L^3 \lambda^4 }{(1-\gamma)^2 \mu } \geq \frac{1}{4}$ or $\alpha \leq \frac{\sqrt{\mu} (1-\gamma) }{2\sqrt{3}  c_1 c_2  L^{3/2} \lambda^2 }$. 
	Therefore, from \eqref{tran_SC_0} 
	\begin{align} \label{tran_SC}
		&		\Ex  \|\bar{e}_x^{k}\|^2  +  \tfrac{c_1^2}{n}	\Ex \|\hat{\x}^{k}\|^2  
		\nonumber \\
		& \leq  (1-\tfrac{\alpha  \mu}{2})^k \tilde{a}_0 + \left\|(I- A)^{-1} b \right\|_1
		\nonumber \\
		&\leq 
		(1-\tfrac{\alpha  \mu}{2})^k \tilde{a}_0 + \tfrac{2  \sigma^2}{\mu n} \alpha  
		\nonumber \\
		& \quad +
		\tfrac{6 c_1^2 c_2^2 (L/\mu) \lambda^4 \sigma^2 + 4 c_1^2 c_2^2 \lambda^4 \sigma^2 }{1-\gamma} \alpha^2 + \tfrac{8 c_1^2 c_2^2  L^2 \lambda^4 \sigma^2}{\mu n (1-\gamma)^2} \alpha^3.
	\end{align}
	Using $(1-\tfrac{\alpha  \mu}{2})^K \leq \exp(-\tfrac{\alpha  \mu}{2} K)$ and \eqref{hat_0_bound}, it holds that
	\begin{align}  \label{final_a_bound_strong_cvx}
		&		\Ex  \|\bar{e}_x^{K}\|^2  +  \tfrac{c_1^2}{n}	\Ex \|\hat{\x}^{K}\|^2  
		\nonumber  \\
		& \leq \exp(-\tfrac{\alpha  \mu}{2} K) (a_0+ \alpha^2 a^{\star}) +   a_1 \alpha + a_2 \alpha^2 + a_3 \alpha^3,
	\end{align}
	where
	\begin{subequations} \label{a0123_constants_strongcvx}
		\begin{align}
			a_0 & \define \Ex  \|\bar{x}^{0}-x^\star\|^2,\quad 	a^{\star} \define  \tfrac{c_1^2 c_2^2  \lambda^4}{(1-\lambda)^2 n} \|\hat{\U}\tran \grad \mathbf{f}(\mathbf{x}^\star) \|^2 \\
			a_1 & \define   \frac{2  \sigma^2}{\mu n} , \quad 
			a_2 \define   \frac{10 c_1^2 c_2^2 L \lambda^4 \sigma^2 }{\mu (1-\gamma)} \\ 
			a_3 & \define  \frac{8 c_1^2 c_2^2  L^2 \lambda^4 \sigma^2}{\mu n (1-\gamma)^2}.
		\end{align}
	\end{subequations}
	Note that by combining all stepsize conditions, it is sufficient to require
	\begin{align} \label{stepsize_strongcvx_proof_all}
		\alpha \leq \frac{1}{\underline{\alpha}} \define  \min\left\{\frac{1-\lambda}{8 L }, \frac{\mu (1-\lambda)}{8 c_1^2 c_2^2 L^2 \lambda^4 }, \frac{\sqrt{\mu} (1-\lambda) }{4\sqrt{3}  c_1 c_2  L^{3/2} \lambda^2 }\right\}.
	\end{align}
	We now select 
	\begin{align} \label{stepsize_strongcvx_rate}
		\alpha=\min \left\{\ln \left(\max \left\{2, \mu^{2} (a_0 + \tfrac{a^{\star}}{\underline{\alpha}^2}) \frac{K}{a_1} \right\}\right)/\mu  K, \tfrac{1}{\underline{\alpha}}\right\} \leq \frac{1}{\underline{\alpha}}.
	\end{align}
	Under this choice the exponential term in \eqref{final_a_bound_strong_cvx} can be upper bounded as follows.  i) If $\alpha = \frac{\ln \left(\max \left\{1, \mu^{2} (a_0 + a^{\star}/\underline{\alpha}^2) K / a_1\right\}\right)}{\mu  K} \leq \frac{1}{\underline{\alpha}} $ then     
	\begin{align*}
		&	\exp(-\tfrac{\alpha  \mu}{2} K) (a_0+ \alpha^2 a^{\star}) \\
		& \leq \tilde{\mathcal{O}}\left( (a_0 + \tfrac{a^{\star}}{\underline{\alpha}^2})  \exp \left[-\ln \left(\max \left\{1, \mu^{2} (a_0 + \tfrac{a^{\star}}{\underline{\alpha}^2}) K / a_1 \right\}\right)\right]\right)
		\nonumber \\
		&=\mathcal{O}\left(\frac{a_1}{\mu  K}\right);
	\end{align*}
	ii) Otherwise $\alpha = \frac{1}{\underline{\alpha}} \leq \frac{\ln \left(\max \left\{1, \mu^{2} (a_0+a^{\star}/\underline{\alpha}^2) K / a_1 \right\}\right)}{\mu  K}$ and
	\begin{align*}
		\exp(-\tfrac{\alpha  \mu}{2} K) (a_0+ \alpha^2 a^{\star}) 
		& = 	   \exp \left[-\tfrac{  \mu K}{2 \underline{\alpha}}\right] (a_0+\tfrac{a^{\star}}{\underline{\alpha}^2}).
	\end{align*}
	Therefore, under the stepsize condition \eqref{stepsize_strongcvx_rate}  it holds that
	\begin{align*} 
		&		\Ex  \|\bar{e}_x^{K}\|^2  +  \tfrac{c_1^2}{n}	\Ex \|\hat{\x}^{K}\|^2  
		\nonumber  \\
		& \leq \exp(-\tfrac{\alpha  \mu}{2} K) (a_0+ \alpha^2 a^{\star}) +   a_1 \alpha + a_2 \alpha^2 + a_3 \alpha^3	  
		\nonumber \\ 
		&\leq \tilde{\mathcal{O}}\left(\frac{a_1}{\mu  K}+\frac{a_2 }{\mu^{2} K^{2}}
		+  \frac{a_3 }{\mu^{3}  K^{3}} + (a_0+\tfrac{a^{\star}}{ \underline{\alpha}^2})  \exp \left[-\tfrac{ K}{\underline{\alpha}}\right] \right).
	\end{align*}
	Plugging the constants \eqref{a0123_constants_strongcvx} into the above inequality, using \eqref{stepsize_strongcvx_proof_all} and \eqref{avg_hat_bound}  yields our rate \eqref{scvx_theorem}.

	\footnotesize
	\bibliographystyle{ieeetr}
	\bibliography{myref_gt}

\end{document}